\documentclass[12pt]{amsart}
\usepackage{graphicx,enumerate,url}
\usepackage{placeins}

\numberwithin{equation}{section}

\usepackage[usenames]{color}

\usepackage[margin=1in]{geometry}

\allowdisplaybreaks

\theoremstyle{plain}

\newtheorem{theorem}{Theorem}[section]

\theoremstyle{definition}

\def\Dir{\operatorname{Dir}}
\def\vMF{\operatorname{vMF}}
\def\Mult{\operatorname{Mult}}

\newcommand{\Mag}[1]{\lvert \lvert #1 \rvert \rvert}
\newcommand{\iid}{\stackrel{\mathrm{iid}}{\sim}}

\newcommand{\etr}{\mathrm{etr}}
\newcommand{\trace}{\operatorname{tr}}

\numberwithin{table}{section}
\numberwithin{figure}{section}
\usepackage{setspace}
\begin{document}

\title{Learning Subspaces of Different Dimensions}
\author[B.~St.~Thomas]{Brian~St.~Thomas}
\address{Department of Statistical Science, Duke University, Durham, NC 27708-0251.}
\email{brian.st.thomas@duke.edu}
\author[L.~Lin]{Lizhen Lin}
\address{Department of Statistics and Data Sciences, The University of Texas at Austin, Austin, TX 78746.}
\email{lizhen.lin@austin.utexas.edu}
\author[L.-H.~Lim]{Lek-Heng Lim}
\address{Computational and Applied Mathematics Initiative, Department of Statistics,
University of Chicago, Chicago, IL 60637-1514.}
\email{lekheng@galton.uchicago.edu}
\author[S.~Mukherjee]{Sayan Mukherjee}
\address{Departments of Statistical Science, Mathematics, Computer Science, Duke University, Durham, NC 27708-0251.}
\email{sayan@stat.duke.edu}

\begin{abstract}
We introduce a Bayesian model for inferring mixtures of subspaces of different dimensions. The key challenge in such a mixture model is specification of prior distributions over subspaces of different dimensions. We address this challenge by embedding subspaces or Grassmann manifolds into a sphere of relatively low dimension and specifying priors on the sphere. We provide an efficient sampling algorithm for the posterior distribution of the model parameters. We illustrate that a simple extension of our mixture of subspaces model can be applied to topic modeling. We also prove posterior consistency for the mixture of subspaces model.
The utility of our approach is demonstrated with applications to real and simulated data.
\end{abstract}
\maketitle

\section{Introduction}

The problem of modeling manifolds has been of great interest in a variety of statistical problems including dimension reduction \cite{misha1,donoho03hessian,lle}, characterizing the distributions of statistical models as points on a Riemannian manifold \cite{Amari82,Efron78,Rao45}, and the extensive literature in statistics and machine learning on manifold learning \cite{Cook2007,GK2006,MukZhouWu}. A generalization of the manifold setting is to model unions and intersections of manifolds (of possibly different dimensions), formally called stratified spaces \cite{BenMukWang2012,Geiger01,GorMac1988}. Stratified spaces arise when data or parameter spaces are characterized by combinations of manifolds such as the case of mixture models. One of the most important special cases arises when the manifolds involved are all  affine subspaces or linear subspaces. Mixtures of linear subspaces have been suggested in applications such as tracking images \cite{HarRanSap2005,VidMaSas2005}, quantitative analysis of evolution or artificial selection
\cite{HansenHoule,Lande}, applications in communication and coding theory \cite{AshikhminCalderbank,ZhengTse},
and is relevant for text modeling \cite{reisinger10,blei03}.In this paper we provide a model for the simplest instance of inferring stratified spaces, estimating mixtures of linear subspaces of different dimensions.

The idea of dimension reduction via projections onto low-dimensional subspaces goes back at least to Adcock \cite{Adcock1878}  and Edgworth \cite{Edgeworth1884}, with methodological and foundational contributions by R.~A.~Fisher \cite{Fisher1922}; see \cite{Cook2007} for an excellent review. It is very interesting that in 1922 Fisher suggested that the statistical setting where the number of variables is greater than the number of observations, $p \gg n$ could be addressed by reducing the dimension of $p$ to very few $p^*$ summaries of the data where $p^* < n.$ The summaries in this setting were linear combinations of the variables. This idea of dimension reduction has been extensively used statistics ranging from classical methods such as principal components analysis (PCA) \cite{hotelling-33} to a variety of recent methods, some algorithmic and some likelihood based, that fall under the category of nonlinear dimension reduction and manifold learning \cite{misha1,Cook2007,donoho03hessian,GK2006,MukZhouWu,lle}. A challenging setting for both algorithmic and probabilistic models in this setting is where the data are being generated from multiple populations inducing a mixture distribution. It is particularly challenging when the mixtures are of different dimensions.

In many applications a useful model for the observed high-dimensional data assumes the data is concentrated around a lower-dimensional structure in the high-dimensional ambient space. In addition, it is often the case that the data is generated from multiple processes or populations each of which has low-dimensional structure. In general, the degrees of freedom or number of parameters of the processes capturing the different populations need not be equal. In this paper, we address this problem of modeling data arising from a mixture of manifolds of different dimensions for the restricted case where the manifolds are linear subspaces.

The most recent work that offers both estimators and provides guarantees on estimates for inferring mixtures of subspaces has been limited to \textit{equidimensional} subspaces \cite{LerZha2010,page13}. A Bayesian procedure for mixtures of subspaces of equal dimensions was developed in  Page et al.\ \cite{page13}. A penalized loss based procedure was introduced in
Lerman and Zhang \cite{LerZha2010} to learn mixtures of $K$-flats. There are significant difficulties in extending either
approach to subspaces of different dimensions. The key difficulty in extending either approach is addressing the singularity introduced in moving between subspaces of different dimensions when one parameterizes a subspace as a point on the Grassmann manifold and uses the natural geodesic on this manifold. This difficulty appears in the Bayesian approach as requiring the posterior samples to come from models of different dimensions which will require methods such as
reversible jump MCMC which may cause mixing problems. The difficulty is immediate in the penalized loss model as the loss is based on a distance to subspaces and if the dimensions of the subspaces vary the loss based procedure becomes very difficult.
A related line of work appears in Hoff \cite{hoff2009} where efficient methodology is developed to sample from a space of orthonormal matrices with fixed intrinsic dimension using the matrix Bingham--von~Mises--Fisher distribution. One of the
results in this paper is a procedure to sample from orthonormal matrices of varying intrinsic dimensions. From a geometric perspective a method was developed in \cite{hoff2009} to simulate from a Stiefel manifold with fixed intrinsic dimension, in 
this paper we provide a methodology to simulate over  Stiefel manifolds of varying intrinsic dimensions.

The key idea we develop in this paper is that subspaces of different dimensions $1,2,\dots,m$ can be embedded into a sphere of relatively low dimension $\mathbb{S}^{(m-1)(m+2)/2}$  where chordal distances on the sphere can be used to
compute distances between subspaces of differing dimensions \cite{Conwayetal}. This embedding removes the discontinuity that occurs in moving between subspaces of different dimensions when one uses the natural metric for a Grassmann manifold. The other tool we make use of is a Gibbs posterior \cite{JiangTanner08} which allows us to efficiently obtain posterior samples of the model parameters. 

The structure of the paper is as follows. In Section~\ref{mixsub} we state a likelihood model for a mixture of $k$ subspaces each of dimension $d_k$. In Section~\ref{conway} we define the embedding procedure we use to model subspaces of different dimensions and specify the model with respect to the likelihood and prior. In Section~\ref{post} we provide an algorithm for sampling from the posterior distribution. For some of the parameters standard methods will not be sufficient for efficient sampling and we use a Gibbs posterior for efficient sampling. In Section ~\ref{topicmodels} we extend the mixture of subspaces model to topic modeling. In Section~\ref{consist} a frequentist analysis of the Bayesian procedure is given that proves posterior consistency of the procedure. In Section~\ref{data} we use simulated data to provide an empirical analysis of the model and then we use real data to show the utility of the model.We close with a discussion.

\section{Model specification}
\label{mixsub}

\subsection{Notation}
We first specify notation for the geometric objects used throughout this paper. The Grassmann manifold or Grassmannian of $d$-dimensional subspaces in $\mathbb{R}^m$ will be denoted $\operatorname{Gr}(d,m)$. The Stiefel manifold of $m \times d$ matrices with orthonormal columns will be denoted $\operatorname{V}(d,m)$ and when $d=m$ we write $\operatorname{O}(d)$ for the orthogonal group. We use boldfaced uppercase letters, e.g., $\mathbf{U}$, to denote subspaces and the corresponding letter in normal typeface, e.g., $U$, to denote the matrix whose columns form an orthonormal basis for the respective subspace. Note that $\mathbf{U} \in \operatorname{Gr}(d,m)$ and $U \in \operatorname{V}(d,m)$. A subspace has infinitely many different orthonormal bases, related to one another by the equivalence relation $U' = UX$ where $X\in \operatorname{O}(d)$. We identify a subspace $\mathbf{U}$ with the equivalence class of all its orthonormal bases $\{ U X \in  \operatorname{V}(m,d) : X \in  \operatorname{O}(d)\}$ thereby allowing the identification  $\operatorname{Gr}(d,m) = \operatorname{V}(d,m)/ \operatorname{O}(d)$.

In this article, the dimension of the ambient space $m$ will always be fixed but our discussions will often involve multiple copies of Grassmannians   $\operatorname{Gr}(d,m) $ with different values of $d$. We will use the term `Grassmannian of dimension $d$' when referring to   $\operatorname{Gr}(d,m) $ even though as a manifold, $\dim \operatorname{Gr}(d,m) = d(m-d)$.

\subsection{Likelihood specification} We consider data $X = (x_1,\dots,x_n)$ drawn independent and identically from a mixture of $K$ subspaces where each observation $x_i$ is measured in the ambient space $\mathbb{R}^m$. We assume that each population is concentrated near a linear subspace $\mathbf{U}_k$ which we represent with an orthonormal basis $U_k$, $\mathbf{U}_k = \operatorname{span}(U_k)$, $k =1,\dots, K$.

We first state the likelihood of a sample conditional on the mixture component. Each mixture component is modeled using a $d_k$-dimensional normal distribution to capture the subspace and a $m-d_k$-dimensional normal distribution to model the residual error or null space:
\begin{displaymath}
U_k^\mathsf{T}x \sim \mathcal{N}_{d_k}(\mu_k, \Sigma_{k}), \quad V_k^\mathsf{T}x \sim \mathcal{N}_{m-d_k}(V_k^\mathsf{T}\theta_k, \sigma^2_k I),
\end{displaymath}
where $U_k$ is the orthonormal basis for the $k$th component and is modeled by a multivariate normal with mean $\mu_k$ and covariance
$\Sigma_k$ and $V_k $ is the basis for the null space $\operatorname{ker}(U_k)$ which models the residual error as multivariate normal with variance
$\sigma^2_k I$. We are estimating affine subspaces so the parameter $\theta_k$ serves as a location parameter for the component and by construction $\theta_k \in V_k$. Also note that without loss of generality we can assume that $\Sigma_{k}$ is diagonal since we may
diagonalize the covariance matrix $\Sigma_{k} = Q_kD_kQ_k^\mathsf{T}$ and rotate $U_k$ by $Q_k$ resulting in a parameterization that depends on
$U_k$ and a diagonal matrix.The distributions for the null space and and subspace can be combined and specified by either of 
the following parameterizations
\begin{equation}
x \sim  \begin{cases}
 \mathcal{N}_m \left(U_k\mu_k + \theta_k, U_k\Sigma_{k}U_k^\mathsf{T} + \sigma^2_kV_kV_k^\mathsf{T}\right) \\
 \mathcal{N}_m \left(U_k\mu_k + \theta_k, U_k(\Sigma_{k}-\sigma_k^2I_{d_k})U_k^\mathsf{T} +\sigma_k^2I_m \right).
\end{cases}
\end{equation}
It will be convenient for us to use the second parameterization for our likelihood model.

Given the above likelihood model for a component we can specify the following mixture model
\begin{equation}
x \sim \sum_{k = 1}^K w_k \, \mathcal{N}_m \left(U_k \mu_k + \theta_k, U_k(\Sigma_{k}-\sigma_k^2I_{d_k})U_k^\mathsf{T} +\sigma_k^2I_m \right),
\end{equation}
where $w = (w_1,\dots,w_K)$ is a probability vector and we assume $K$ components. We will use a latent or auxiliary variable approach
to sample from the above mixture model and specify a $K$-dimensional vector $z$ with a single entry of $1$ and all other entries of zero, $\delta \sim \operatorname{Mult}(1,w)$. The conditional probability of $x$ given the latent variable is
$$x\mid \delta \sim \sum_{k=1}^K \delta_k \mathcal{N}_m \left( U_k\mu_k + \theta_k, U_k(\Sigma_{k}-\sigma_k^2I_{d_k})U_k^\mathsf{T} +\sigma_k^2I_m\right).$$

\subsection{Prior specification and the spherical embedding}
\label{conway}

The parameters in the likelihood for each component are $(\theta_k,\Sigma_k,\sigma_k^2,U_k,\mu_k,d_k)$ and the mixture parameters are weights $w$. Again we fix the number of mixtures at $K$. Prior specification for some of these parameters
are straightforward. For example the location parameter $\theta_k$ is normal, the variance terms $\Sigma_k$ and $\sigma_k^2$ are Gamma, and the mixture weights are Dirichlet. A prior distribution for each triple $(U_k,\mu_k,d_k)$ is less obvious.

The inherent difficulty in sampling this triple is that we do not want to fix the dimension of the subspace $d_k$, we want to consider $d_k$ as random. We can state the following joint  prior on the triple $(U_k,\mu_k,d_k)$
\[
\pi(U_k,\mu_k,d_k) = \pi(U_k \mid d_k) \, \pi(\mu_k \mid d_k) \, \pi(d_k).
\]
Given $d_k$ we can specify $\mu_k \mid d_k$ as a multivariate normal of dimension $d_k$. Given $d_k$ we can also specify a conjugate distribution for $U_k$ as the  von~Mises--Fisher (MF) distribution
\[
\operatorname{MF}(U_k\mid A) \propto \etr(A^\mathsf{T}U_k),
\]
where $\etr$ is the exponential trace operator.  The matrix von Mises--Fisher distribution is a spherical distribution over the set of all $m\times d_k$ matrices, also known as the Stiefel manifold which we denote as $\operatorname{V}(d_k,m)$. A prior on $d_k$ would take values over $[0,\dots,m]$ and for each value the conditional
distributions $\pi(U_k \mid d_k)$ and  $\pi(\mu_k \mid d_k)$ need to be specified. For $\mu_k$ a prior distribution of $\mathcal{N}_{d_k} (0, \lambda I)$ seems reasonable since we can assume the mean is zero and the entries independent for any $d_k$. Specifying the conditional distribution for $\pi(U_k \mid d_k)$ is not as clear. As $d_k$ changes the dimension of the matrix $A$ needs to change and one cannot simply add columns of zeroes since columns need to be orthonormal. An additional constraint on the prior is that a small change in dimension $d_k$ should only change the prior on  $U_k$ slightly. This constraint is to avoid model fitting inconsistencies. This constraint highlights  the key difficulty in prior specification over subspaces of different dimensions: how to measure the distance between subspaces of different dimensions. Note that we can not simply integrate out $d_k$ or $U_k$ as nuisance parameters since we have no prior specification.

We will use the geometry of the subspace $\mathbf{U}_k$ to specify an appropriate joint prior on $(U_k, d_k)$. Recall that the set of all $d_k$-dimensional linear subspaces in $\mathbb{R}^m$ is the Grassmann manifold $\operatorname{Gr}(d_k,m)$ and that we represent a subspace $\mathbf{U}_k \in \operatorname{Gr}(d_k,m)$ with an orthonormal matrix $U_k \in \operatorname{V}(d_k,m)$ from an equivalence class $\{U_k \in \operatorname{V}(k,m) : \operatorname{span}(U_k) = \mathbf{U}_k \}$. We need to place priors on Grassmanians of different dimension $d_k$. The key tool we use to specify such a prior is the embedding of $\operatorname{Gr}(d_k, m)$ into $\mathbb{S}^{(m-1)(m+2)/2}$, an appropriately chosen sphere\footnote{Note that the dimension of a sphere in $\mathbb{R}^d$ is $d-1$ and that $m(m+1)/2 -1 = (m-1)(m+2)/2$.} in $\mathbb{R}^{m(m+1)/2}$, as proposed in Conway et al.\ \cite{Conwayetal}.
This embedding allows us to embed  subspaces of different dimensions into the same space and measure distances between the embedded subspaces as a function of only the ambient (embedded) space. We will use this embedding to place a prior on $U_k$ which implicitly specifies a prior on $d_k$. This embedding will have some very nice properties in terms of prior specification and computation.

The following theorem states that embedding the Grassmanian into a sphere allows us to measure distances between subspaces.
\begin{theorem}[Conway--Hardin--Sloane 1996]
The representation of a subspace $\mathbf{U} \in \operatorname{Gr}(d,m)$ by its projection matrix $P_\mathbf{U}$ gives an isometric embedding of $\operatorname{Gr}(d,m)$ into a sphere of radius $\sqrt{d(m-d)/m}$ in $\mathbb{R}^{m(m+1)/2}$, with $d_p(\mathbf{U},\mathbf{V}) = \frac{1}{\sqrt{2}}\lVert P_\mathbf{U}-P_\mathbf{V} \rVert_F$, where $P_\mathbf{V}$ is the projection matrix onto $\mathbf{V}$.
\end{theorem}

The embedding procedure proceeds in the following steps: (1) given a basis $U_k$ compute the projection matrix $P_k = U_k^\mathsf{T} U_k$, (2)
take all the entries of $P_k$ in the upper triangle (or lower triangle) as well as all the elements in the diagonal except for one as a vector in
$\mathbb{R}^{m(m+1)/2-1}$. The sum of all the entries on the vector will be a constant, this is a result of the orthogonality of $U_k$, which means that all the subspaces of dimension $k$ lie on the same sphere. The key observation by Conway et al.\ \cite{Conwayetal} was that if the extra coordinate is included, thus embedding into $\mathbb{R}^{m(m+1)/2}$, the subspaces are still embedded into spheres and each of these spheres are cross sections of a higher-dimensional sphere which we denote as $\mathbb{S}^{(m-1)(m+2)/2}$. The sphere $\mathbb{S}^{(m-1)(m+2)/2}$ is centered at
$\varphi\left(\frac{1}{2}I_m \right) = \operatorname{vech}\left(\frac{1}{2}I_m\right)$ where $\varphi( A)$
denotes the embedding of the projection matrix $A$ and $\operatorname{vech}$ is the half-vectorization operation
\[
 \operatorname{vech} \left(\begin{bmatrix} a &b \\ b & d \end{bmatrix} \right) = \begin{bmatrix} a \\ b \\ d \end{bmatrix}.
\]
The $0$-dimensional subspace is embedded at the origin $\mathbf{0} \in \mathbb{R}^{m(m+1)/2}$. The radius of  $\mathbb{S}^{(m-1)(m+2)/2}$ is $\sqrt{{m(m+1)/8}}$. In summary,
\[
\mathbb{S}^{(m-1)(m+2)/2} = \{ x \in \mathbb{R}^{m(m+1)/2} : \lVert x - c \rVert^2 = m(m+1)/8 \}, \quad \text{where } c = \operatorname{vech}\left(\frac{1}{2}I_m\right).
\]
Grassmann manifolds
are embedded into cross-sections of $\mathbb{S}^{(m-1)(m+2)/2}$ where the projection matrix corresponding to the pre-image has an integer valued trace.
The geodesic distance along the surface of the sphere, $d_{\mathbb{S}^{(m-1)(m+2)/2}}$, corresponds to the projective distance $d_p(\,\cdot\, , \,\cdot\,)$ between two subspaces $\mathbf{U}_1,\mathbf{U}_2 \in \operatorname{Gr}(d,m)$,
\[
d_p(\mathbf{U}_1,\mathbf{U}_2) =\left[\sum\nolimits_{j =1}^d \sin^2(\theta_j)\right]^{1/2},
\]
where $\theta_{1},\dots,\theta_{d}$ are the principal angles between the
subspaces. We illustrate the embedding for two projection matrices in
Figure~\ref{fig1}.
\begin{figure}[h!]
\begin{center}
\includegraphics[scale=.32]{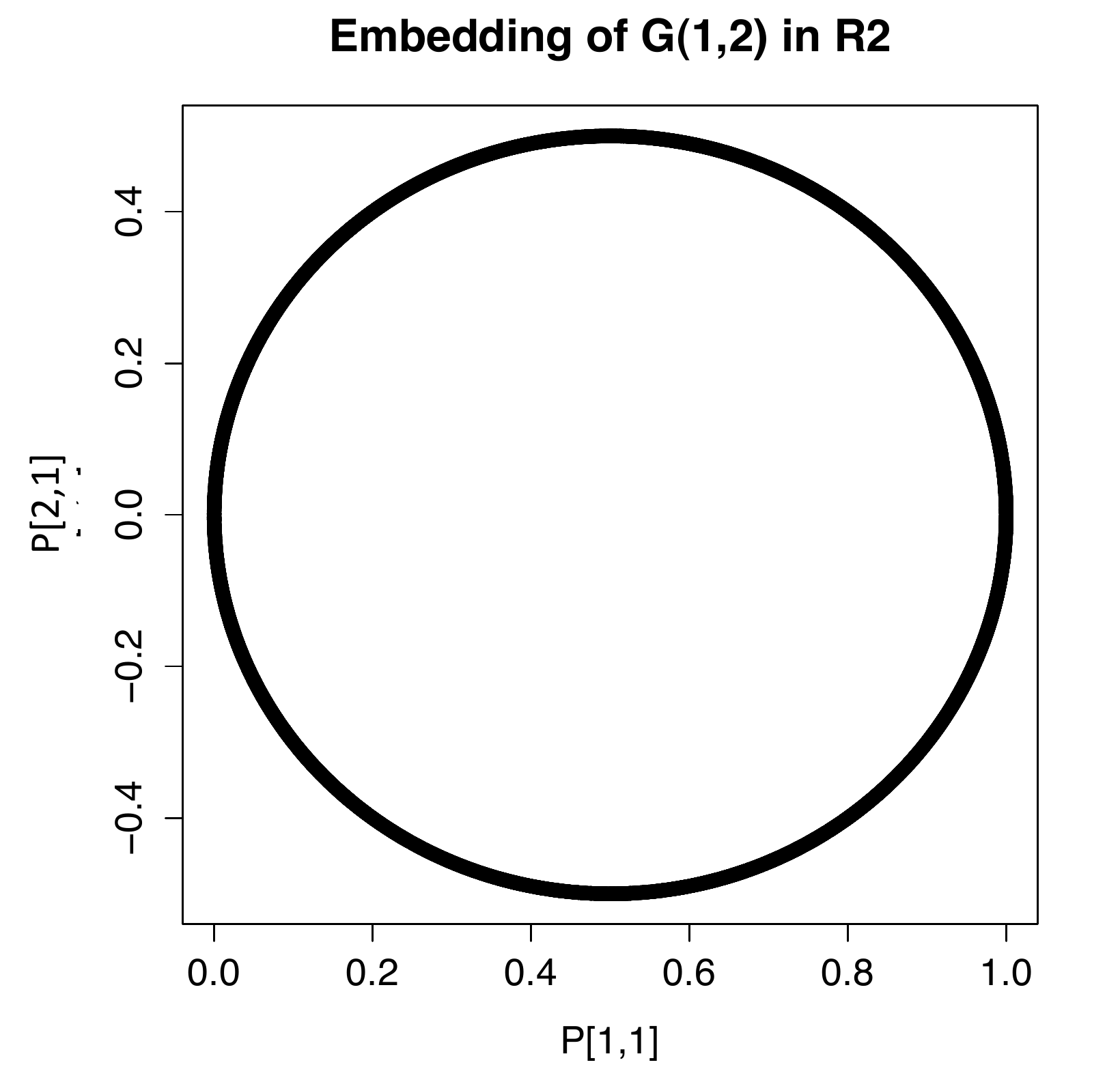} \includegraphics[scale=.32]{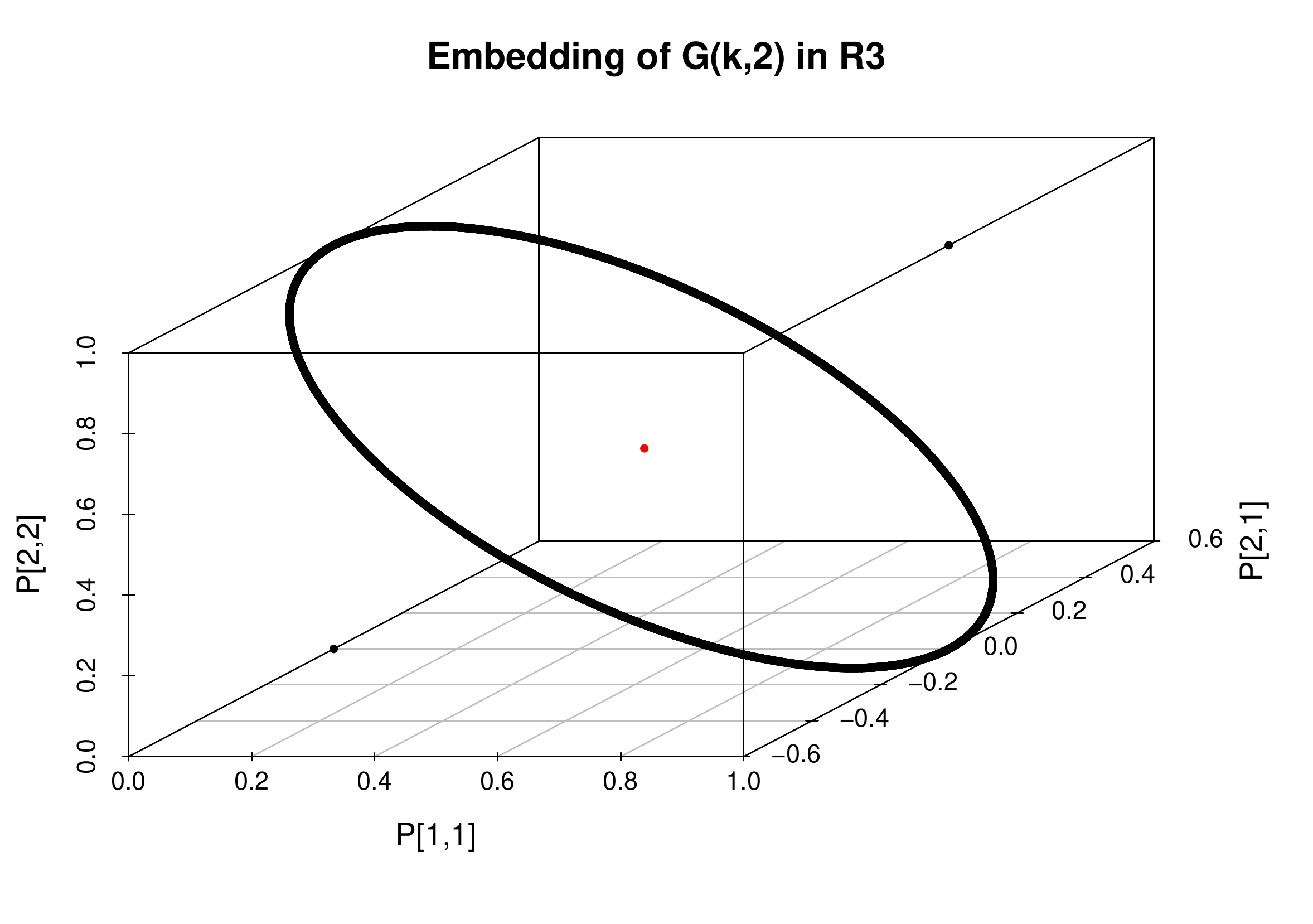}
\end{center}
\caption{An illustration of the spherical embedding for subspaces from $\operatorname{Gr}(1,2)$ into $\mathbb{R}^2$ on the left, and of $\operatorname{Gr}(k,2)$, $k=0,1,2$ on the right. Images of the embedding are in black, and the center of the sphere is in red. The coordinates for the embedding into $\mathbb{R}^2$ is the first column of the projection matrix. By including the last entry in the diagonal of the projection matrix, we obtain coordinates for the embedding into $\mathbb{R}^3$. $\operatorname{Gr}(0,2)$ and $\operatorname{Gr}(2,2)$ are trivial sets giving $\mathbf{0}_3$ and $I_3$ as their projection matrices. They act as poles on the sphere with coordinates $(0,0,0)$ and $(1,0,1)$.  }
\label{fig1}
 \end{figure}

The representation of Grassmannians as points on $\mathbb{S}^{(m-1)(m+2)/2}$ has several useful properties.
\begin{description}
\item[Sphere interpretation] The sphere $\mathbb{S}^{(m-1)(m+2)/2}$ provides an intuitive way to sample subspaces of different dimensions by sampling from $\mathbb{S}^{(m-1)(m+2)/2}$. Under the projective distance, the sphere also has an intuitive structure. For example, distances between subspaces of different dimensions can also be computed as the distance between points on the sphere, these points will be on different cross-sections. Under the projective distance, the orthogonal complement of a subspace $\mathbf{U}$ is the point on
$\mathbb{S}^{(m-1)(m+2)/2}$ that maximizes the projective distance. Further, the projection matrix is always invariant to the representation $U$.

\item[Differentiable] The projective distance however is square differentiable everywhere, making it more suitable in general for optimization problems. This is not the case for distances like geodesic distance or the Asimov--Golub--Van~Loan distance where maximizing the distance between a set of subspaces will result in distances that lie near non-differentiable points  \cite{Conwayetal}. This numerical instability can lead to sub-optimal solutions.

\item[Ease of computation] The projective distance is easy to compute via principal angles, which are in turn readily computable with singular value decomposition \cite{GV2013}. Working with the embedding requires only a relatively small number of coordinates --- in fact only quadratic in $m$ or  $m(m+1)/2$. Furthermore one can exploit many properties of a sphere in Euclidean space in our computations. For example
sampling from a sphere is simple. The number of required coordinates is small compared to alternative embeddings of the Grassmannian,
see \cite{HamLee2008}. In contrast the usual Pl\"{u}cker embedding requires a number of coordinates that is  ${m}\choose{d}$, i.e., exponential in $m$. Moreover the Pl\"{u}cker embedding does not reveal a clear relationship between Grassmannians of different dimensions, as there is using the spherical embedding.
\end{description}

We will place a prior on projection matrices by placing a distribution
over the lower half of $\mathbb{S}^{(m-1)(m+2)/2}$, points on $\mathbb{S}^{(m-1)(m+2)/2}$ corresponding to
cross-sections where the subspace corresponding to the pre-image has
dimension $d <  {m(m+1)/4}$. We only consider the lower half since we assume
the model to be low-dimensional. The prior over projection matrices
imples a prior over $U_k$ and $d_k$. A point drawn from $\mathbb{S}^{(m-1)(m+2)/2}$ may not
correspond to a subspace, recall only points with integer trace have
subspaces as a pre-image. We address this problem by the following procedure:
given a sampled point $q \in \mathbb{S}^{(m-1)(m+2)/2}$ we return the closest point $p \in
\mathbb{S}^{(m-1)(m+2)/2}$ that is the pre-image of a subspace. The following theorem states the procedure.

\begin{theorem}\label{project}
Given a point $q\in \mathbb{S}^\ell$, the point $p$ that minimizes the geodesic distance on $\mathbb{S}^\ell$, $d_{\mathbb{S}^\ell}(q,p)$, subject to
$$\varphi^{-1}(p) \in \bigcup_{d=0}^\ell \operatorname{Gr}(d,\ell)$$ can be found by the following procedure
\begin{enumerate}[\upshape (i)]
\item Compute $Q = \varphi^{-1}(q)$.
\item Set the dimension of $p$ to $d = \trace(Q)$.
\item Compute the eigendecomposition $Q = A\Lambda A^{-1}$.
\item Set $B$ an $\ell \times d$ matrix equal to the columns of $A$ corresponding to the top $d$ eigenvalues.
\item Let $p=\varphi (BB^\mathsf{T})$.
\end{enumerate}
\end{theorem}

\begin{proof} In the case where the point $q\in \mathbb{S}^\ell$ is already on a cross-section of the sphere corresponding to $\operatorname{Gr}(d,\ell)$, the eigendecomposition will return exactly $d$ non-zero eigenvalues. The eigenvectors give a basis for the subspace that is embedded into the point $q$. Similarly when the point $q$ is between cross sections corresponding to Grassmannians, the above algorithm minimizes the Euclidean distance between the point $p$ and $q$, and therefore minimizes the distance on $\mathbb{S}^\ell$.
\end{proof}

The full model is specified as follows for each $x_i$, $i=1,\dots,n$,
\begin{align}
w &\sim \operatorname{Dir}_K(\alpha), \nonumber \\
\delta_i &\sim \operatorname{Mult}(w), \nonumber \\
P_k & \sim \mathcal{P}(\mathbb{S}^{(m-1)(m+2)/2}), \quad U_k U_k^\mathsf{T}=P_k, \quad d_k = \trace(P_k),  \label{uni} \\
\mu_k \mid d_k & \sim  \mathcal{N}_{d_k}(0,\lambda I), \nonumber \\
\theta_k \mid  U_k &\sim \mathcal{N}_m(0, \phi I), \quad U_k^\mathsf{T}\theta_k = 0, \label{orthreq}\\
\sigma_k^{-2}  &\sim  \operatorname{Ga}(a,b), \nonumber \\
\Sigma_{k(j)}^{-1}\mid d_k & \sim  \operatorname{Ga}(c,d), \quad j = 1,\dots,d_k, \nonumber \\
x_i \mid \delta_i &\sim \sum\nolimits_{k=1}^K \delta_{ik} \mathcal{N}_m\Bigl(U_k\mu_k + \theta_k, U_k(\Sigma_{k}-\sigma_k^2 I_{d_k})U_k^\mathsf{T} +\sigma_k^2I_m \Bigr), \nonumber
\end{align}
where equation \eqref{uni} corresponds to sampling from a distribution $\mathcal{P}$ supported on the lower half of the sphere $\mathbb{S}^{(m-1)(m+2)/2}$
a projection matrix $P_k$ that corresponds to a subspace and computing the dimension $d_k$ as the trace of the subspace and computing the subspace $U_k$ from the projection and equation \eqref{orthreq} corresponds to sampling from a normal distribution subject to the projection constraint $U_k^\mathsf{T} \theta_k = 0$.

\section{Posterior sampling}
\label{post}

In this section we provide an efficient algorithm for sampling the model parameters from the posterior distribution. Sampling directly from a joint posterior distribution of all the parameters is intractable and we will use Markov chain Monte Carlo methods for sampling. For most of the parameters we can sample from the posterior using a Gibbs sampler. This is not the case for sampling from the posterior distribution over projection matrices with prior  $\mathcal P$ on the sphere $\mathbb{S}^{(m-1)(m+2)/2}$. The prior $\mathcal P$ should place more mass on cross-sections of the sphere corresponding to lower dimensions $d_k$. Sampling efficiently from a joint distribution on $d_k, P_k$ is difficult. We will address this problem by using a Gibbs posterior  \cite{JiangTanner08} to sample the projection matrices. We first state the Gibbs posterior we use to sample $U_k$ and $\theta_k$ efficiently and the rationale for this form of the posterior. We then close with the sampling algorithm for all the model parameters.

It is not obvious how to place a prior on the sphere $\mathbb{S}^{(m-1)(m+2)/2}$ that will allow for efficiently sampling. We can however follow the idea of a Gibbs posterior to design an efficient sampler. The idea behind a Gibbs posterior is to replace the standard posterior which takes the form of
\[
\mathit{posterior} \propto  \mathit{prior} \times  \mathit{likelihood}
\]
with a distribution based on a loss or risk function that depends on both the data as well the parameter of interest in our case the loss function
is given by
\begin{align}
L(P_{[1:K]}, \theta_{[1:K]},X) &= \frac{1}{n}\sum\nolimits_{i=1}^n\Bigl[\min_{k=1,\dots,K} \Bigl(\lVert P_k(x_i - \theta_k) - (x_i -\theta_k) \rVert^2 +  \operatorname{tr}(P_k)\Bigr)\Bigr],  \label{loss}\\
 &= \frac{1}{n}\sum\nolimits_{i=1}^n\Bigl[\min_{k=1,\dots,K}  ( e_{ik} + d_k ) \Bigr], \nonumber
 \end{align}
where $e_{ik}$ is the residual error for the $i$th sample given by the $k$-th subspace with the error defined by our likelihood model. The above loss function corresponds to computing for each sample the residual error to the closest subspace weighted by the dimension of the subspace. The penalty weighting the dimension of the subspace enforces a prior that puts more mass on subspaces of lower dimension. Given the likelihood or loss function we state the following Gibbs posterior
\begin{equation}
\label{gibbs}
g(P_{[1:K]},\theta_{[1:K]}\mid X) \propto \exp\bigl(-n\psi L(P_{[1:K]}, \theta_{[1:K]},X)\bigr) \pi(P_{[1:K]}) \pi(\theta_{[1:K]}),
\end{equation}
where $\psi$ is a chosen temperature parameter. A Gibbs posterior is simply a loss oriented alternative to the likelihood based posterior distribution. Traditionally it is used to account for model misspecification. Here the Gibbs posterior is used to avoid overfitting by arbitrarily increasing the dimension of the subspace and for computational efficiency in sampling.

\subsection{Sampling $U_{[1:K]}$ and $\theta_{[1:K]}$ from the Gibbs posterior} In this subsection we outline our procedure for sampling from the model parameters $U_{[1:K]}$ and $\theta_{[1:K]}$ using a Metropolis--Hastings algorithm which is effectively a random walk on the sphere. We first state a few facts that we will use.
First recall that there is a deterministic relation between $U_k$ and $P_k$, so given a $P_k$ we can compute $U_k$. Also recall that a point sampled from $\mathbb{S}^m$ is not the pre-image of a subspace. Given a point $s^0_k \in \mathbb{S}^m$ we denote the subspace corresponding to this point as $P_k = \varphi^{-1}(s^{0}_k)$, this is the closest projection matrix to $s^0_k$ corresponding to a subspace. The procedure to compute $P_k$ from $s^{0}_k$ is given in Theorem \ref{project}. We obtain $U_k$ correspond to the top $d_k$ eigenvectors of $P_k$ where
$d_k$ is the trace of $P_k$.

We now state two procedures. The first procedure initializes the parameters $U_{[1:K]}$ and $\theta_{[1:K]}$.
The second procedure computes the $\ell$-th sample of the parameters.

The first procedure which we denote as $\mathbf{Initialize}(U_{[1:K]},\theta_{[1:K]})$ proceeds as follows:
\begin{enumerate}[\upshape 1.]
\item Draw $\sigma \sim \mathfrak{S}_K$, the symmetric group of permutations on $K$ elements.
\item For $i=1,\dots,K$,
\begin{enumerate}[\upshape (a)]
\item draw $z_{\sigma(i)}^0 \sim \mathcal{N}_{m(m+1)/2}(0,\tau I)$;
\item compute $s_{\sigma(i)}^0 = (\sqrt{{m(m+1)/8}})z_{\sigma(i)}^0/\lVert z_{\sigma(i)}^0 \rVert + \varphi(I_m)$;
\item compute $P_{\sigma(i)}^0 = \varphi^{-1}(s_{\sigma(i)}^0)$;
\item compute $d_{\sigma(i)}^0 = \operatorname{tr}(P_{\sigma(i)}^0)$;
\item compute $U_{\sigma(i)}^0$ as the top $d_{\sigma(i)}^0$ eigenvectors of $P_{\sigma(i)}^0$;
\item draw $\beta_{\sigma(i)}^0 \sim \mathcal{N}(0,I_m)$;
\item compute $\theta_{\sigma(i)}^0 = (I_m -P_{\sigma(i)}^0) \beta_{\sigma(i)}^0$.
\end{enumerate}
\end{enumerate}
The first step permutes the order we initialize the $K$ components. Step~(a) samples a point from a multivariate normal with the dimension of the sphere. In Step~(b) we normalize the sampled point, recenter it, and embed it onto the sphere $\mathbb{S}^ {m(m+1)/2}$. In Step~(c) we compute the projection matrix by computing the closest subspace to the embedded point computed in Step~(b). Given the projection matrix we compute the dimension in Step~(d) and the basis of the subspace in Step~(e). Steps~(e) and (f) we compute the $\theta$ parameters.

The second procedure which we denote as $\mathbf{Update}\bigl(U^{(\ell)}_{[1:K]},\theta^{(\ell)}_{[1:K]}\bigr)$ computes the $\ell$-th sample as follows:
 \begin{enumerate}[\upshape 1.]
\item Draw $\sigma \sim \mathfrak{S}_K$, the symmetric group of permutations on $K$ elements.
\item For $i=1,\dots,K$,
\begin{enumerate}[\upshape (a)]
\item draw $z_{\sigma(i)} \sim
  \mathcal{N}_{m(m+1)/2}(z_{\sigma(i)}^{(\ell-1)} , \tau I)$;
\item compute $s_{\sigma(i)} = (\sqrt{{m(m+1)/8}})z_{\sigma(i)}/\lVert z_{\sigma(i)} \rVert + \varphi(I_m)$;
\item compute $P_{\sigma(i)} = \varphi^{-1}(s_{\sigma(i)})$;
\item compute $d_{\sigma(i)} = \operatorname{tr}(P_{\sigma(i)})$;
\item compute $U_{\sigma(i)}$ as the top $d_{\sigma(i)}$ eigenvectors of $P_{\sigma(i)}$;
\item draw $u \sim \operatorname{Unif}[0,1]$;
\item set
\[
P_{[1:K]} = \bigl[P^{(\ell-1)}_{[1:K] - {\sigma(i)}} , P_{\sigma(i)}\bigr];
\]
\item set
\[
\theta_{[1:K]} = \bigl[\theta^{(\ell-1)}_{[1:K] - {\sigma(i)}} , (I_m-U_{\sigma(i)} U_{\sigma(i)}^{T})\theta_{\sigma(i)}^{(\ell-1)}\bigr];
\]
\item compute the acceptance probability
\[
\alpha = \frac{\exp\bigl(-n\psi L(P_{[1:K]},\theta_{[1:K]},X)\bigr)}{\exp\bigl(-n\psi L(P^{(\ell-1)}_{[1:K]},\theta^{(\ell-1)}_{[1:K]},X)\bigr)};
\]
\item set
\[
\bigl(U_{\sigma(i)}^{(\ell)},z_{\sigma(i)}^{(\ell)}\bigr) =
\begin{cases}
\bigl(U_{\sigma(i)}, z_{\sigma(i)}\bigr) &\text{if } \alpha >u,  \\
\bigl(U_{\sigma(i)}^{(\ell-1)},z_{\sigma(i)}^{(\ell-1)}\bigr) &\text{otherwise};
\end{cases}
\]
\item draw $\beta_{\sigma(i)} \sim \mathcal{N}_m(\beta_{\sigma(i)}^{(\ell-1)},I_m)$;
\item compute $\theta_{\sigma(i)} = (I_m -P_{\sigma(i)}^{\ell-1}) \beta_{\sigma(i)}$;
\item draw $u \sim \operatorname{Unif}[0,1]$;
\item set
\[
\theta_{[1:K]} = \bigl[\theta^{(\ell-1)}_{[1:K] - {\sigma(i)}},\theta_{\sigma(i)} \bigr];
\]
\item compute the acceptance probability
\[
\alpha = \frac{\exp\bigl(-n\psi L(P^{(\ell-1)}_{[1:K]},\theta_{[1:K]},X)\bigr)}{\exp\bigl(-n\psi L(P^{(\ell-1)}_{[1:K]},\theta^{(\ell-1)}_{[1:K]},X)\bigr)};
\]
\item set
\[
\bigl(\theta_{\sigma(i)}^{(\ell)},\beta_{\sigma(i)}^{(\ell)}\bigr) =
\begin{cases}
\bigl(\theta_{\sigma(i)}, \beta_{\sigma(i)}\bigr) &\text{if } \alpha >u , \\
\bigl(\theta_{\sigma(i)}^{(\ell-1)},\beta_{\sigma(i)}^{(\ell-1)}\bigr) & \text{otherwise}.
\end{cases}
\]
\end{enumerate}
\end{enumerate}
Many steps of this procedure are the same as the first procedure with the following exceptions. In Steps~(a) and (k) we are centering the random walk to the previous values of $z_{\sigma(i)}$ and $\beta_{\sigma(i)}$ respectively. Step~(g) updates the set of
$K$ projection matrices by replacing the $i$-th projection matrix in the set with the proposed new matrix. Step~(h) is analogous to Step~(g) but for the set of $\theta$ vectors. In Step~(j) we update the subspace and in Step~(p) we update the $\theta$ vector.

\subsection{Sampling algorithm} We now state the algorithm we use to sample from the posterior. To simplify notation we work with precision
matrices $J_k = \Sigma_k^{-1}$ instead of the inverse of covariance matrices for each mixture component. Similarly, we work with the precision of the $k$-th component $\gamma_k$ instead of the  inverse of the variance,  $\gamma_k = \sigma_k^{-2}$.

The follow procedure provides posterior samples:
  \begin{enumerate}[\upshape 1.]
  \item Draw $U_{[1:K]}^{(0)},\theta_{[1:K]}^{(0)}, d_{[1:K]}^{(0)} \sim \mathbf{Initialize}(U_{[1:K]},\theta_{[1:K]})$.
  \item Draw $J_{k(j_k)} \sim \operatorname{Ga}(a,b)$ for $k=1,\dots,K$ and $j_k = 1,\dots,d_k^{(0)}$.
  \item For $t=1,\dots,T$,
  \begin{enumerate}[\upshape (a)]
  \item for $i=1,\dots,n$ and $k=1,\dots,K$, compute
   \[
      e_{ik} =  \bigl\lVert P^{(t-1)}_k\bigl(x_i - \theta^{(t-1)}_k\bigr) - \bigl(x_i -\theta^{(t-1)}_k\bigr) \bigr\rVert^2;
   \]
  \item for  $i=1,\dots,n$, set
   \[
     w_i = \biggl(\frac{\exp(- \kappa r_{i1})}{\sum_{j'=1}^K\exp(- \kappa r_{ij'})},\dots,  \frac{\exp(- \kappa r_{iK})}{\sum_{j'=1}^K\exp(- \kappa r_{ij'})} \biggr);
   \]
  \item  for  $i=1,\dots,n$, draw $\delta_i \sim \operatorname{Mult}(w_i)$;
  \item  update for $k=1,\dots,K$ each $\mu_k^{(t)} \sim \mathcal{N}(m^*_k,S^*_k)$ where
  \[
      S^*_k=\bigl(n_k J^{(t-1)}_k+\lambda^{-1} I \bigr)^{-1}, \quad m^*_k = S^*_k\Bigl( U^{(t-1)\mathsf{T}}_k J_k^{(t-1)} \sum\nolimits_{\delta_i = k} x_i \Bigr),
   \]
    and $n_k = \#\{ i : \delta_i = k\}$;
  \item update for $k=1,\dots,K$, and each $\gamma_k^{(t)} \sim \operatorname{Ga}(a_k^*,b^*_k)$,
  \begin{align*}
  a^*_k &= n_k(m-d_k)+a, \\
   b^*_k &=   b+\frac{n_k}{2} (\theta^{(t-1)}_k)^\mathsf{T} \theta^{(t-1)}_k + \sum\nolimits_{\delta_i = k}\Bigl( \frac{1}{2} x_i^\mathsf{T} x_i  -
  x_i^\mathsf{T}   U_k^{(t-1)} U_k^{(t-1)\mathsf{T}}  x_i \Bigr) -  \theta_k^{(t-1)\mathsf{T}} \sum\nolimits_{\delta_i = k} x_i;
  \end{align*}
   \item update for $k=1,\dots,K$, and  $j_k = 1,\dots,d_k^{(t)}$,
   \[
      J_{k (j_k)}^{(t)} \sim \operatorname{Ga}\Bigl( \frac{n_k}{2} +a, b + \frac{1}{2} \sum\nolimits_{\delta_i = k} \bigl(U_k^{(t-1)\mathsf{T} }x_i - \mu_k\bigr)_{j_k}^2 \Bigr),
   \]
   where $(u)_j$ denotes the $j$th element of the vector $u$;
   \item draw
    \[
       U_{[1:K]}^{(t)},\theta_{[1:K]}^{(t)}, d_{[1:K]}^{(t)} \sim \mathbf{Update}\bigl(U^{(t-1)}_{[1:K]},\theta^{(t-1)}_{[1:K]}\bigr).
    \]
   \end{enumerate}
  \end{enumerate}
The update steps for $\mu, \sigma^2, \Sigma$ are (d), (e), (f) respectively and are given by the conditional probabilities given all other variables. Steps~(a), (b), and (c) assign the latent membership variable to each observation based on the distance to the $K$ subspaces.
We set the parameter $\kappa$ very large which effective assigns membership of each $x_i$ to the subspace with least residual error.


When drawing from the Gibb's posterior distribution via a
Metropolis--Hastings algorithm, the proposal distribution and temperature are adjusted
through a burn-in period. In the first stage of burn-in, the proposal variance parameter $\tau=1$ is fixed, while temperature is selected by a decreasing line search on a log-scale grid, from $10^{-20}$ to $10^{20}$ until the acceptance ratio reaches the 10\%--90\% range. With temperature fixed, the proposal variance $\tau$ is adjusted until the acceptance ratio falls in the 25\%--45\% range during the burn-in period. Thinning was applied in that every third draw of the sampler was kept, this was determined
from autocorrelation analysis.

\section{Specification of  a Topic Model}
\label{topicmodels}
\subsection{Generative Model on the Stiefel manifold}

The idea behind topic modeling is to specify a generative model for documents where the model parameters provide some intuition about a collection of documents. A common representation for documents is what is called a "bag of words" model
where the grammar and  structure of a document is ignored and a document is just a vector of counts of words \cite{blei03,deer90,hoffman99}. A natural generative model for collections of documents is an admixture of topics where each topic is a multinomial distribution over words, this model is called a latent Dirichlet allocation (LDA) model \cite{blei03,pritchard00}. We will propose a slight variation of the LDA model later in this section which is a direct extension/application of a mixture of subspaces.

We first state the standard LDA model. Given $D$ documents, $k$ topics, and a vocabulary of size $V$ the counts of
the $i$-th word in the $d$-th document is specified by the following hierarchical model
\begin{displaymath}
  \begin{array}{rcll}
  \theta_d \mid \alpha & \sim& \Dir_K(\alpha)  & \mbox{(topic probabilities for each document)},\\
\phi_k \mid \beta &\sim &\Dir_V(\beta) & \mbox{(word probabilities for each topic)}, \\
z_{i,d} \mid \theta_d &\sim & \Mult_K(\theta_d) & \mbox{(topic assignment for each word in each document)}, \\
w_{i,d} \mid \phi_{z_{i,d}} &\sim & \Mult(\phi_{z_{i,d}}) & \mbox{(word counts for each word in each document)}.
 \end{array}.
\end{displaymath}

In a spherical admixture model (SAM) \cite{reisinger10} the vector of word counts in each document transformed by centering at zero and normalizing to unit length.  The idea behind a SAM is to represent data as direction distributions on a hypersphere.  
The advantage of a SAM is that one can simultaneously model both frequency as well as  presence/absence of words, an LDA model can only model frequency. There is empirical evidence of greater accuracy in using a SAM for sparse data such as text \cite{bannerjee05,zhong05}. We extend the SAM model in two important ways, first by ensuring all the topics are orthogonal. The logic behind orthogonality constraints in the topics is to avoid the empirically observed problem of redundant topics.
Strategies to eliminate this problem include removing what are called stop words from the corpus,  for example words including ``such", ``as", and ``and." However, it is not always the case that stop words for a particular corpus are known a priori. For example the word ``topic" should be a stop word in a corpus of papers on topic modeling. Introducing an orthogonality constraint on the topics can enforce the prior knowledge that they should be interpretable as distinct. In \cite{hoff2009} a Bayesian model
for an orthogonal SAM is specified and a posterior sampling procedure is developed. A key insight this paper was how to efficiently simulate from the set of orthonormal matrices using the matrix Bingham--von~Mises--Fisher distribution. For an orthogonal SAM model the $K$ topics on $V$ words were modeled using the matrix Bingham--von~Mises--Fisher distribution which is on the Stiefel manifold, $\mathcal{V}(V,K)$.

Our second extension is to infer the number of topics $K$. While the orthogonality constraints help with interpretation of topics and removes redundant topics there are still topics with low posterior mixture probabilities and low coherence can still occur. This is mainly driven by  misspecification of the  number of topics. We now state our novel SAM model that enforces orthonormal columns as well as allows for the inference of the number of topics. The novel contributions of our model are inference of the number of topics by using the geometry of the Conway embedding to place a joint prior over the number of topics and topic.
We are able to sample a from Stiefel manifolds of variable intrinsic dimension $K$ by coupling draws from the von~Mises--Fisher distribution with inversion of the Conway embedding. This allows us to avoid using the matrix Bingham--von~Mises--Fisher distribution.

We provide some intuition for our SAM with orthogonality constraints as well as useful notation before we specify the model.
We will simulate a topic matrix $\phi$ where the columns of the matrix are topics $\{\phi_k\}_{k=1}^K$ and the number of topics $K$ is random, this orthogonal matrix is sampled from a distribution over Stiefel manifolds, $\mathcal{V}(V,K)$, with fixed ambient dimension $V$ (the number of words) and variable embedding dimension $K$. For each document
probability vector of topic proportions $\theta_d$ over the $K$ topics is generated. The $\ell_2$ normalized
unit vector $v_d$ representing normalized word frequencies for each document is generated from the topic
proportions $\theta_d$ and the topic matrix $\phi$. The following notation and concepts will be used in the generative model.
We denote $\mathbb{S}$ as the Conway sphere $\mathbb{S}^{(V-1)(V+2)/2}$ this is the collection of orthogonal subspaces
of variable dimension embedded into a sphere. We denote  $\varphi(\cdot)$ and  $\varphi^{-1}(\cdot)$ as the embedding 
function and its inverse respectively. We denote $\vMF_{\mathbb{S}}$ as the von Mises--Fisher distribution over
the Conway sphere  $\mathbb{S}^{(V-1)(V+2)/2}$ and $\vMF_{V}$ as the von Mises--Fisher distribution over the 
unit sphere $\mathbb{S}^{V}$. Given the topics matrix and the topic proportions $\theta_d$ for a document a
spherical average of the topics with respect to the proportions is the admixed parameter that models the combination of
topics in document and is computed by
$$\mbox{avg}(\phi, \theta_d) =  \frac{\phi\theta_d}{\Mag{\phi\theta_d}}.$$
We did not use the Buss-Fillmore spherical average due to computational considerations, we wanted to avoid iterative procedures. Given a vocabulary of size $V$ and $D$ documents, the $\ell_2$ normalized unit vector $v_d$ for
each document is specified by the following hierarchical model
\begin{equation}
\label{SAMmod}
  \begin{array}{rcll}
\mu \mid \kappa_0 & \sim& \vMF_{\mathbb{S}}(m, \kappa_0)  & \quad \mbox{(corpus average)},\\
\eta \mid \mu, \xi  &\sim &\vMF_{\mathbb{S}}(\mu, \xi) & \quad  \mbox{(embedded orthogonal topics)}, \\
(\phi, K) &=& \varphi^{-1}(\eta) & \quad  \mbox{(orthogonal topics and number of topics)}, \\
\theta_d \mid \alpha &\sim& \Dir_K(\alpha) & \quad \mbox{(topic proportions for each document)}, \\
\bar{\phi}_d  & =& \mbox{avg}(\phi, \theta_d) & \quad \mbox{(spherical average, admixed parameter)}, \\
v_d \mid \bar{\phi}_d, \kappa &\sim& \vMF_V\left(\bar{\phi}_d, \tau\right) &  \quad \mbox{(generates a document vector)}.
  \end{array}
\end{equation}
The main difference in the above model and prior models on spheres \cite{hoff2009,reisinger10} is that instead of sampling topics from the embedding on the Conway sphere a fixed $K$ topic vectors were each sampled from a von Mises--Fisher distribution over $\mathbb{S}^{V}$. In \cite{reisinger10} the vectors were not constrained to be orthogonal, in \cite{hoff2009} an efficient procedure is given to simulate these $K$ orthogonal vectors. The Conway sphere in this case is extremely high dimensional and there is computational utility in reducing the vocabulary size.  
 
As in the mixture of subspaces model we require a prior that will place greater weight on models with fewer topics, as increasing the number of topics will result in a better fit with respect to the likelihood. In the same spirit as Section \ref{post} we specify a Gibbs posterior to place a prior on the Conway sphere that can be efficiently be sampled from and favors models with fewer topics. We specify the following loss function for each document vector $v_d$
\begin{equation}
L\big(\phi,K \mid \{v_d\}_{d=1}^D, \tau\big) = \frac{2 D K}{V}-\sum_{d=1}^D \left(\tau\bar{\phi}_d^T v_d  \right)
\end{equation}
and corresponding Gibbs posterior
\begin{equation}
g(\phi, K \mid  \{v_d\}_{d=1}^D, \tau) \propto \exp \Big(-\psi D  L \times \big(\phi,K \mid \{v_d\}_{d=1}^D, \tau\big)  \Big)\pi(\phi)
\end{equation}
The maximum penalty above is $D$ and would counterbalance a perfect fit to each document with a loss value of zero.
Using the Gibbs posterior allows us to skip the step of estimating the corpus average parameter $\mu$ since the remaining parameters and $\mu$ are conditionally independent given the topics. We set the temperature parameter $\psi$ using out-of-sample fits on a random search over $\log(\psi) \in [-10,10]$. 

Inference of the remaining parameters of Model \eqref{SAMmod} are estimated using the same sampling steps as in a standard SAM, once the topics and number of topics are sampled.  The high-dimension of the Conway sphere can result in slower mixing of the topics and it is of interest to explore EM or Hamiltonian Monte Carlo approaches for computational gain.

\section{Posterior consistency}
\label{consist}

In this section we show that the mixture of subspaces model specified in Section \ref{mixsub} has good frequentist properties.
We provide an asymptotic analysis of our model and state some theoretical guarantees. We will prove posterior consistency---the posterior distribution specified by our model contracts to any arbitrary neighborhood of the true data generating density. 
Adapting proof techniques from the extensive literature on posterior consistency of Bayesian models \cite{Ghoshal10} to
our model is non-trivial.

Before we state our results we provide an explanation of the relation between our consistency result, the estimation procedure,
and ideally what proof statements would be of interest. The MCMC algorithm stated in Section \ref{post} uses a Gibbs posterior while our analysis in this section is for a standard posterior. Providing an argument for posterior consistency and proper calibration of the Gibbs posterior is of great interest but beyond the scope of this paper. A natural question is the asymptotic analysis of the convergence of mixture components and weights, basically an analysis of the clustering performance. The inference of mixture components is significantly more complicated that convergence in density and involves subtle identifiability issues. 

Let $\mathcal{M}$ be the space of all the densities in $\mathbb{R}^m$ and  $f_0$ be the true data generating density. We first define some notion of distances and neighborhoods in $\mathcal{M}$.
A weak neighborhood  of $f_0$ with radius $\epsilon$ is defined as
\begin{equation}
\label{eq-weaknb}
W_{\epsilon}(f_0)=\left\{f: \left\lvert \int gf\, dx- \int gf_0\, dx\right\rvert \leq \epsilon \; \text{for all } g\in C_b(\mathbb{R}^m) \right\},
\end{equation}
where $C_b(\mathbb{R}^m)$ is the space of all continuous and bounded functions on $\mathbb{R}^m$.
The  Hellinger  distance $d_H(f,f_0)$ is defined as
\[
d_H(f,f_0) = \left(\frac{1}{2}\int \bigl[\sqrt{f(x)}-\sqrt{f_0(x)}\bigr]^2\, dx\right)^{1/2}.
\]
Denote  $U_{\epsilon}(f_0)$   an $\epsilon$-Hellinger neighborhood around
$f_0$ with respect to $d_H$.  The Kullback--Leibler (KL)  divergence between $f_0$ and $f$  is defined to be
\begin{align}
\label{eq-KLdivergence}
d_{\operatorname{KL}}(f_0,f)=\int  f_0(x) \log \frac{ f_0(x)}{f(x)}\, dx,
\end{align}
with $K_{\epsilon}(f_0)$ denoting an $\epsilon$-KL neighborhood of $f_0$.

One of the key geometric insights of the Conway embedding given in Section 2 is that the geometric embedding allows for
the specification of prior distributions of the parameters  on  a smooth space. Let $\Pi_s$ be a prior on the sphere $\mathbb{S}^{(m-1)(m+2)/2}$ which can be taken to be the uniform distribution or the von~Mises--Fisher distribution.  By projecting the samples from $\Pi_s$ onto the cross-sections of the sphere, $\Pi_s$ induces a prior distribution on the subspaces basis $\mathbf{U}$ which we denote by $\Pi_{\mathbf{U}}$.


 Note that our model induces a prior $\Pi$ on $\mathcal{M}$. Assume the true density $f_0$ follows the following the regularity conditions
 \begin{enumerate}[\upshape (i)]
 \item $f_0(x) $ is bounded away from zero and bounded above by some constant $M$ for all $x\in \mathbb R^m$;
 \item $\lvert \int \log (f_0(x))f_0(x)\, dx \rvert <\infty$;
 \item for some $\delta>0$, $\int [\log ( f_0(x)/f_{\delta}(x))] f_0(x)\, dx<\infty$, where $f_{\delta}(x)=\inf_{y: \lvert y-x\rvert <\delta}f_0(y)$;
 \item there exists $\alpha>0$ such that $\int \lvert x\rvert^{2(1+\alpha)m}f_0(x)\, dx<\infty$.
 \end{enumerate}
 We will show that the posterior distribution $\Pi(\,\cdot\, \mid x_1,\dots, x_n)$ concentrates around any true density $f_0$ as $n\rightarrow \infty$. The following theorem is on weak consistency.
\begin{theorem}
\label{th-Wconsistency}
The posterior distribution $\Pi(\,\cdot\, \mid x_1,\dots, x_n)$ is weakly consistent. That is,
 for all $\epsilon>0$,
\begin{equation}
\label{eq-consistency}
\Pi(W_{\epsilon}(f_0) \mid x_1,\dots, x_n)\rightarrow 1\quad Pf_0^{\infty}\text{-almost surely as } n\rightarrow \infty,
\end{equation}
where $W_{\epsilon}(f_0)$ is a weak neighborhood of $f_0$ with radius $\epsilon$ and $Pf_0^{\infty}$ represents the true probability measure for $(x_1, x_2,\dots)$.
\end{theorem}

\begin{proof}
A result due to \cite{schwarz78} states that if $\Pi$ assumes positive mass to any Kullback--Leibler neighborhood of $f_0$, then the resulting posterior is weakly consistent. Therefore, one needs to show for all $\epsilon>0$,
\begin{equation}
\Pi(\{f : d_{\operatorname{KL}}(f_0, f)\leq \epsilon\})>0.
\end{equation}


Note that $\Pi(d_i=m)>0$ for any $i=1,\dots, K$. Then one has
\begin{multline}
\label{eq-kl2}
\Pi(K_{\epsilon}(f_0)) \ge \\
\int _{\operatorname{O}(m)\times\dots\times \operatorname{O}(m)}\Pi(K_{\epsilon}(f_0) \mid U_1,\dots, U_K,\; d_1=\dots= d_k=m)\, d\Pi(U_1,\dots, U_K  \mid d_1=\dots= d_k=m).
\end{multline}
Therefore it suffices to show that $\Pi(K_{\epsilon}(f_0) \mid U_1,\dots, U_K )>0$ where $U_1,\dots, U_K$ are the bases of the respective $m$-dimensional subspaces $\mathbf{U}_1,\dots, \mathbf{U}_K$.


We will show that there exists $K$ large enough such that given $m$-dimensional subspaces $\mathbf{U}_1,\dots, \mathbf{U}_K$, the following mixture model assigns positive mass to any KL neighborhood of $f_0$,
\begin{equation}
\label{eq-finitemodel}
f(x,\mathbf{U},\Sigma)=\sum\nolimits_{j=1}^K w_j \mathcal{N}(\phi(\mu_j),\widetilde\Sigma_j),
\end{equation}
with $\phi(\mu_j)=U_j\mu_j+\theta_j$ and $\widetilde\Sigma_j=U_j(\Sigma_{j}-\sigma^2I_{d_j})U_j^\mathsf{T} +\sigma_j^2I_m.$

If $\mathbf{U}_1,\dots, \mathbf{U}_K$ have the same dimension $m$ and $U_1,\dots, U_K$ are a choice of orthonormal bases on the respective subspaces, then an infinite-dimensional version of our model can be given by
\begin{align*}
X\sim g(x,\Gamma)=\int_{\mathbb{R}^m}\mathcal{N}(x;\phi(\mu),\Sigma)P(d\mu),
\phi(\mu)=U\mu+\theta,\quad \Sigma=U(\Sigma_0-\sigma^2I_m)U^\mathsf{T}+\sigma^2I_m
\end{align*}
with parameters $\Gamma=(U,\theta,\Sigma_0, \sigma, P)$. The prior for $P$ can be given by a Dirichlet process prior whose base measure has full support in $\mathbb{R}^m$ while the priors of the rest  parameters $(U,\theta,\Sigma_0, \sigma)$ can be given the same as those of our model. By Theorem 3.1 of \cite{page13} or Theorem 2 of \cite{wu2010}, there exists an open subset $\mathcal{P}$ of the space of all the probability measures on $\mathbb R^m$ such that for  for all $\epsilon>0$, any $P\in\mathcal{P}$ such that
\begin{align}
\label{eq-ineq0}
\int_{\mathbb{R}^m} f_0(x)\log \frac{f_0(x)}{g(x,\Gamma)}\, dx\leq \epsilon.
\end{align}
We will first show that for any $\epsilon'>0$  there exists $K$ large enough and $w_1,\dots,w_K$, $\theta_1,\dots, \theta_K$, $\mu_1,\dots, \mu_K$, $U_1,\dots, U_K$, $\Sigma_1,\dots, \Sigma_K$, $\sigma_1,\dots, \sigma_K$, and $\sigma^2$ such that
\begin{align*}
\left\lvert g(x,\Gamma)-\sum\nolimits_{j=1}^K w_j \mathcal{N}(\phi(\mu_j),\widetilde\Sigma_j)\right\rvert \leq \epsilon',
\end{align*}
for any $P\in\mathcal P$.

Let $L$ be some large enough number. We first partition $\mathbb R^m$ into  $L^m+1$ sets.  Let
\[
A_{i_1,\dots, i_m} =\prod_{j=1}^m\left(-L+(i_j-1)\frac{\log L}{L}, -L+i_j\frac{\log L}{L}\right]
\]
and
\[
A_0^L\cup \left( \bigcup\nolimits_{i_1,\dots, i_m =1}^L A_{i_1,\dots, i_m} \right)=\mathbb R^m.
\]
Pick $\mu_j\in A_{i_1,\dots, i_m} $ where $j=1,\dots, L^m$ and let $w_j=P(A_{i_1,\dots, i_m}) $. Approximating the integral by the finite sum over the cubes in $\mathbb R^m$, for all $\epsilon''>0$, there exists $L$ large enough such that
\begin{align*}
\left \lvert g(x,\Gamma)-\sum\nolimits_{j=1}^{L^m} w_j \mathcal{N}(\mu_jU+\theta_j, \Sigma)\right\rvert \leq \epsilon''.
\end{align*}
Let $K=L^m$.  Now let $U_1,\dots, U_{K}$ be $K$ points in the $\epsilon''$ neighborhood of $U$, $\theta_1,\dots, \theta_K$ in the $\epsilon''$ neighborhood of $\theta$, $\Sigma_1,\dots, \Sigma_K$ be in the $\epsilon''$ neighborhood of $\Sigma_0$, and $\sigma_1,\dots, \sigma_K$ in the $\epsilon''$ neighborhood of $\sigma$, then by the continuity of $\sum_{j=1}^K w_j \mathcal{N}(\mu_jU+\theta, \Sigma)$, one can show that
\begin{align*}
\left\lvert g(x,\Gamma)-\sum\nolimits_{j=1}^{K} w_j \mathcal{N}(\mu_jU_j+\theta_j, \Sigma_j)\right\rvert &\leq \left\lvert g(x,\Gamma)-\sum\nolimits_{j=1}^K w_j \mathcal{N}(\mu_jU+\theta, \Sigma)\right\rvert \\
&\qquad +\left\lvert \sum\nolimits_{j=1}^K w_j \mathcal{N}(\mu_jU+\theta, \Sigma)-\sum\nolimits_{j=1}^K w_j \mathcal{N}(\mu_jU_j+\theta_j,\widetilde \Sigma_j)\right\rvert \\
&\leq \epsilon''+\delta(\epsilon''),
\end{align*}
where $\delta(\epsilon'')\rightarrow 0$  when $\epsilon''\rightarrow 0$.
For the above choices of $w_1,\dots,w_K$, $\mu_1,\dots, \mu_K$, $\theta_1,\dots, \theta_K$, $U_1,\dots, U_K$, $\Sigma_1,\dots, \Sigma_K$, $\sigma_1,\dots, \sigma_K$ corresponding to any $P\in \mathcal{P}$,  one looks at
\begin{multline*}
\int_{\mathbb{R}^m} f_0(x)\log \frac{f_0(x)}{\sum_{j=1}^K w_j \mathcal{N}(\mu_jU_j+\theta_j, \widetilde\Sigma_j)}\, dx
= \int_{\mathbb{R}^m} f_0(x)\log \frac{f_0(x)}{g(x,\Gamma)}\, dx \\
+\int_{\mathbb{R}^m} f_0(x)\log \frac{g(x,\Gamma)}{\sum_{j=1}^K w_j \mathcal{N}(\mu_jU_j+\theta_j,\widetilde \Sigma_j)}\, dx.
\end{multline*}
Take $\epsilon''=\epsilon$. By the continuity of the $\log$ function, one has
\begin{align}
\label{eq-ineq}
\int_{\mathbb{R}^m} f_0(x)\log \frac{f_0(x)}{\sum_{j=1}^K w_j \mathcal{N}(\mu_jU_j+\theta_j, \widetilde\Sigma_j)}\, dx\leq \epsilon+\delta'(\epsilon)
\end{align}
where $\delta'(\epsilon)\rightarrow 0$  when $\epsilon\rightarrow 0$.
Note that our prior assigns positive mass to arbitrary neighborhood of $w_1,\dots,w_K$, $\theta_1,\dots, \theta_K$, $\mu_1,\dots, \mu_K$,
$U_1,\dots, U_K$, $\Sigma_1,\dots, \Sigma_K$, $\sigma_1,\dots, \sigma_K$, $\sigma^2$. By \eqref{eq-ineq0}, \eqref{eq-ineq} and the continuity of the model, our assertion follows.
\end{proof}

In proving the following strong consistency theorem, we assume that  the parameters $\sigma_i^2, \sigma^2$  and the diagonal elements of $\Sigma_i$ ($i=1,\dots, K$) follow i.i.d.\ truncated Gamma priors supported on some bounded interval $[0, M]$ for some large enough constant $M$.
\begin{theorem}
The posterior distribution $\Pi(\,\cdot\, \mid x_1,\dots, x_n)$ is strongly consistent. That is,
 for all $\epsilon>0$,
\begin{equation}
\label{eq-consistency-2}
\Pi\left(U_{\epsilon}(f_0) \mid x_1,\dots, x_n\right)\rightarrow 1\quad Pf_0^{\infty}-\text{almost surely as } n\rightarrow \infty,
\end{equation}
where $U_{\epsilon}(f_0)$ is a  neighborhood of $f_0$ with radius $\epsilon$ with respect to the Hellinger distance.
\end{theorem}

\begin{proof}
By Theorem~\ref{th-Wconsistency}, the true density $f_0$ is in the weak support of our model. Then by a result due to \cite{barron88} (also see Theorem 2 in \cite{Ghoshal99}), for all $\epsilon>0$ if we can construct sieves  $D_{\delta, n}\subseteq \mathcal{M}$ with $\delta<\epsilon$  such that the  metric entropy $\log N(\delta, D_{\delta, n})\leq n\beta$ for some $\beta<\epsilon^2/2$ and $\Pi(D_{\epsilon, n}^c )\leq C_1\exp(-n\beta_1)$ with some constants $C_1$ and $\beta_1$. Then the posterior distribution  $\Pi(\,\cdot\, \mid x_1,\dots, x_n)$ is strongly consistent at $f_0.$

Denote
\[
\Theta=(K, \theta_1,\dots, \theta_K, w_1,\dots,w_K,\mu_1,\dots, \mu_K,\allowbreak U_1,\dots, U_K, \Sigma_1,\dots, \Sigma_K, \allowbreak\sigma_1,\dots, \sigma_K, \boldsymbol{\sigma} )
\]
and 
\[
\boldsymbol{\sigma}=(\sigma_{11},\dots, \sigma_{1d_1},\dots, \sigma_{K1},\dots, \sigma_{Kd_K},\sigma, \sigma_1,\dots, \sigma_K )
\]
where $\sigma_{ij}^2$ is the $j$th diagonal element of $\Sigma_i$ which we assume to be diagonal in our model.
Let
\[
\boldsymbol{\Theta}_n=\{\Theta: K= K_n,\; \lVert \theta_i\rVert \leq \bar{\theta}_n, \; \lVert \mu_i\rVert \leq \bar{\mu}_n, \; \max(\boldsymbol{\sigma})\leq M,\; \min(\boldsymbol{\sigma})\geq h_n,\; i=1,\dots, K\}
\]
where $K_n$, $\bar{\theta}_n $, $\bar{\mu}_n$, $M$ and $h_n$ are some sequences depending on $n$.

Let $D_{\delta, n}= \{f(x,\Theta) : \Theta\in \boldsymbol{\Theta}_n\}$  where $f(x,\Theta)$ is given by \eqref{eq-finitemodel} for any $\Theta$. We need to verify the metric entropy and the prior mass of $D_{\delta, n}$. Let
\[
\Theta^i_{K_n}=( d_1^i,\dots, d_K^i, \theta_1^i,\dots, \theta_K^i, w_1^i,\dots,w_K^i, \mu_1^i,\dots, \mu_K^i, U_1^i,\dots, U_K^i, \boldsymbol{\sigma}^{i}), \quad i=1,2.
\]
Note that posterior consistency with respect to the Hellinger distance is equivalent to posterior consistency with respect to the $L^1$-distance due to the equivalence of the two distances.
For $i = 1, 2$, let $\phi^i(\mu_j)=U_j^i\mu_j^i+\theta_j^i$ and $\widetilde\Sigma_j^i=U_j^i(\Sigma_{j}^i-(\sigma^i)^2I_{d_j^i})(U_j^{i})^\mathsf{T} +(\sigma_j^i)^2I_m.$
One has
\begin{align*}
\int_{\mathbb{R}^m} &\lvert f(x, \Theta^1_{K_n})-f(x, \Theta^2_{K_n})\rvert \, dx\\
&=\int_{\mathbb{R}^m}\left\lvert \sum\nolimits_{j=1}^{K_n}w_j^{1}\mathcal{N}_m(\phi^1(\mu_j),\widetilde\Sigma_j^1)- \sum\nolimits_{j=1}^{K_n}w_j^{2}\mathcal{N}_m(\phi^2(\mu_j),\widetilde\Sigma_j^2)    \right\rvert \, dx\\
&=\int_{\mathbb{R}^m}\Bigl\lvert \sum\nolimits_{j=1}^{K_n}w_j^{1}\mathcal{N}_m(\phi^1(\mu_j),\widetilde\Sigma_j^1)- \sum\nolimits_{j=1}^{K_n}w_j^{2}\mathcal{N}_m(\phi^2(\mu_j),\widetilde\Sigma_j^2)\\
&\qquad \qquad + \sum\nolimits_{j=1}^{K_n}w_j^{1}\mathcal{N}_m(\phi^2(\mu_j),\widetilde\Sigma_j^2)- \sum\nolimits_{j=1}^{K_n}w_j^{1}\mathcal{N}_m(\phi^2(\mu_j),\widetilde\Sigma_j^2)   \Bigr\rvert \, dx\\
&\leq \int_{\mathbb{R}^m}\left\lvert \sum\nolimits_{j=1}^{K_n}w_j^{1}\left( \mathcal{N}_m(\phi^1(\mu_j),\widetilde\Sigma_j^1)- \mathcal{N}_m(\phi^2(\mu_j),\widetilde\Sigma_j^2) \right)\right\rvert \, dx\\
&\qquad \qquad +\int_{\mathbb{R}^m}\left\lvert \sum\nolimits_{j=1}^{K_n}(w_j^{1}-w_j^{2}) \mathcal{N}_m(\phi^2(\mu_j),\widetilde\Sigma_j^2)\right\rvert \, dx\\
&\leq \int_{\mathbb{R}^m}\left\lvert \sum\nolimits_{j=1}^{K_n}w_j^{1}\left( \mathcal{N}_m(\phi^1(\mu_j),\widetilde\Sigma_j^1)- \mathcal{N}_m(\phi^2(\mu_j),\widetilde\Sigma_j^2) \right)\right\rvert \, dx+\sum\nolimits_{j=1}^{K_n}\lvert w_j^{1}-w_j^{2}\rvert \\
&\leq \sum\nolimits_{j=1}^{K_n}w_j^{1}\int_{\mathbb{R}^m}\left\lvert \left( \mathcal{N}_m(\phi^1(\mu_j),\widetilde\Sigma_j^1)- \mathcal{N}_m(\phi^2(\mu_j),\widetilde\Sigma_j^2) \right)\right\rvert \, dx+\sum\nolimits_{j=1}^{K_n}\lvert w_j^{1}-w_j^{2}\rvert .
\end{align*}
Note that
\begin{align*}
\int_{\mathbb{R}^m}&\left\lvert \mathcal{N}_m(\phi^1(\mu_j),\widetilde\Sigma_j^1)- \mathcal{N}_m(\phi^2(\mu_j),\widetilde\Sigma_j^2)\right\rvert \, dx\\
&=\int_{\mathbb{R}^m}\left\lvert \mathcal{N}_m(\phi^1(\mu_j),\widetilde\Sigma_j^1)- \mathcal{N}_m(\phi^2(\mu_j),\widetilde\Sigma_j^1)+\mathcal{N}_m(\phi^2(\mu_j),\widetilde\Sigma_j^1)- \mathcal{N}_m(\phi^2(\mu_j),\widetilde\Sigma_j^2) \right\rvert \, dx\\
&\leq \int_{\mathbb{R}^m}\left\lvert \mathcal{N}_m(\phi^1(\mu_j),\widetilde\Sigma_j^1)- \mathcal{N}_m(\phi^2(\mu_j),\widetilde\Sigma_j^1)\right\rvert \, dx+\int_{\mathbb{R}^m}\left\lvert \mathcal{N}_m(\phi^2(\mu_j),\widetilde\Sigma_j^1)- \mathcal{N}_m(\phi^2(\mu_j),\widetilde\Sigma_j^2)\right\rvert \, dx.\\
\end{align*}
By the proof of Lemma~5 of \cite{wu2010}, one has for the first term of the above expression
\begin{align*}
 \int_{\mathbb{R}^m}\left\lvert \mathcal{N}_m(\phi^1(\mu_j),\widetilde\Sigma_j^1)- \mathcal{N}_m(\phi^2(\mu_j),\widetilde\Sigma_j^1)\right\rvert \, dx
&\leq\sqrt{\frac{2}{\pi}}\frac{\lVert \phi^1(\mu_j)-\phi^2(\mu_j)\rVert}{\lambda_1(\widetilde\Sigma_j^1 )^{1/2}}\\
&=\sqrt{\frac{2}{\pi}}\frac{\lVert(U_j^1\mu_j^1+\theta_j^1)-(U_j^2\mu_j^2+\theta_j^2)  \rVert}{\lambda_1(\widetilde\Sigma_j^1 )^{1/2}}\\
&\leq \sqrt{\frac{2}{\pi\lambda_1(\widetilde\Sigma_j^1)}}\left( \lVert\theta_j^1-\theta_j^2  \rVert+\lVert U_j^1\mu_j^1-U_j^2\mu_j^2  \rVert\right),
\end{align*}
where $\lambda_1(\widetilde\Sigma_j^1)$ is the smallest eigenvalue of $\widetilde\Sigma_j^1$.
Therefore combining all the terms above
\begin{align*}
\int_{\mathbb{R}^m} &\lvert f(x, \Theta^1_{K_n})-f(x, \Theta^2_{K_n})\rvert \, dx\\
&\leq \sum\nolimits_{j=1}^{K_n}\lvert w_j^{1}-w_j^{2}\rvert+\max_{j=1,\dots, K_n} \left\{ \sqrt{\frac{2}{\pi\lambda_1(\widetilde\Sigma_j^1)}}\left( \lVert\theta_j^1-\theta_j^2  \rVert+\lVert U_j^1\mu_j^1-U_j^2\mu_j^2  \rVert\right) \right\}\\
&\qquad \qquad +\sum_{j=1}^{K_n}w_j\int_{\mathbb{R}^m}\left\lvert \mathcal{N}_m(\phi^2(\mu_j),\widetilde\Sigma_j^1)- \mathcal{N}_m(\phi^2(\mu_j),\widetilde\Sigma_j^2)\right\rvert \, dx.
\end{align*}

Without loss of generality, assume $\det(\widetilde\Sigma_j^2) \geq \det(\widetilde\Sigma_j^1)$. One has
\begin{align}
\label{eq-Sigma}
\int_{\mathbb{R}^m}&\left\lvert \mathcal{N}_m(\phi^2(\mu_j),\widetilde\Sigma_j^1)- \mathcal{N}_m(\phi^2(\mu_j),\widetilde\Sigma_j^2)\right\rvert \, dx \nonumber\\ 
&=\frac{1}{(2\pi)^{m/2}}\int_{\mathbb{R}^m}\biggl\lvert\frac{1}{\det(\widetilde\Sigma_j^1)^{1/2}}\exp\bigl[-\tfrac{1}{2} (x- \phi^2(\mu_j))^\mathsf{T}(\widetilde\Sigma_j^1)^{-1}(x- \phi^2(\mu_j)) \bigr] \nonumber\\ 
&\qquad \qquad-\frac{1}{\det(\widetilde\Sigma_j^2)^{1/2}}\exp\bigl[-\tfrac{1}{2} (x- \phi^2(\mu_j))^\mathsf{T}(\widetilde\Sigma_j^2)^{-1}(x- \phi^2(\mu_j)) \bigr]  \biggr\rvert \, dx  \nonumber\\
&\leq \frac{2^{m+1}}{(2\pi)^{m/2}}\int_{[0,\infty)^m}\max\biggl\{0, \; \frac{1}{\det(\widetilde\Sigma_j^1)^{1/2}}\exp\bigl[-\tfrac{1}{2} (x- \phi^2(\mu_j))^\mathsf{T}(\widetilde\Sigma_j^1)^{-1}(x- \phi^2(\mu_j)) \bigr]\\ 
&\qquad \qquad-\frac{1}{\det(\widetilde\Sigma_j^2)^{1/2}}\exp\bigl[-\tfrac{1}{2} (x- \phi^2(\mu_j))^\mathsf{T}(\widetilde\Sigma_j^2)^{-1}(x- \phi^2(\mu_j)) \bigr]  \biggr\}\,dx.  \nonumber
\end{align}

We take $\delta=\epsilon/2$.  We first partition $[0,1]^{K_n}$ into  $N_w$ grid points. With a choice of the grid points given in  Lemma~1 in \cite{Ghoshal99},   the number grid points needed for $\sum\nolimits_{j=1}^{K_n}\lvert w_j^{1}-w_j^{2}\rvert \leq \epsilon/8$ is bounded by
\begin{align*}
\log N_{w}\leq K_n\biggl[  1+\log \biggl(\frac{1+\epsilon/8}{\epsilon/8} \biggr)\biggr].
\end{align*}
Note that the number of balls of radius $\delta R$ used to cover a Euclidean ball centered at the origin of radius $R$ in $\mathbb{R}^m$ is bounded by $(3/\delta)^m$. Therefore, the number of balls of radius $\epsilon[2\pi\lambda_1(\widetilde{\Sigma}_1 )]^{1/2}/32$ needed to cover $\lVert\theta\rVert\leq \bar{\theta}_n$ is bounded by $(96 \bar{\theta}_n/\epsilon  [2\pi\lambda_1(\widetilde{\Sigma}_1 )]^{1/2})^m$. Then one can always find $\theta_1^i,\dots,\theta_{K_n}^i$, $i=1,2$, such that
\begin{align*}
\max_{j=1,\dots, K_n}\sqrt{\frac{2}{\pi\lambda_1(\widetilde\Sigma_j^1)}} \lVert\theta_j^1-\theta_j^2\rVert \leq \frac{\epsilon}{8}.
\end{align*}

Note that $\mu_{j}^1$ and $\mu_j^2$ do not necessarily have the same dimension. We view them as elements in $\mathbb R^m$ by filling the last $m-d_j^i$, $i=1,2$, elements with $0$. Then one can show that  $\lVert U_j^1\mu_j^1-U_j^2\mu_j^2  \rVert\leq \lVert \mu_j^1-\mu_j^2  \rVert. $ Then the covering number for $\mu$ is bounded by $(96 \bar{\mu}_n/\epsilon  [2\pi\lambda_1(\widetilde{\Sigma}_1 )]^{1/2})^m$ with which one can find $\mu_1^i, \dots,\mu_{K_n}^i$, $i=1,2$, such that
\begin{align*}
\max_{j=1,\dots, K_n}\sqrt{\frac{2}{\pi\lambda_1(\widetilde\Sigma_j^1)}} \lVert U_j^1\mu_j^1-U_j^2\mu_j^2  \rVert \leq \frac{\epsilon}{8}.
\end{align*}

Note that for $\Theta\in \boldsymbol\Theta_n$, all the eigenvalues of $\widetilde{\Sigma}$ lie in the interval $[h_n^2, M^2]$. We divide the $m$-dimensional cube $[h_n^2, M^2]^m$ into smaller cubes such that there are $N_{\boldsymbol\lambda}$ grid points for the $m$ eigenvalues $\boldsymbol{\lambda}=(\lambda_1,\dots, \lambda_m)$  of $\widetilde{\Sigma}$.  Note that for any $\widetilde{\Sigma}_1$ whose eigenvalues fall into one of the cubes in the grid, one can always find $\widetilde{\Sigma}_2$ such that
$\widetilde{\Sigma}_1^{-1}-\widetilde{\Sigma}_2^{-1}$ is positive definite.
For example, one may take $\widetilde{\Sigma}_2^{-1}=\widetilde{\Sigma}_1^{-1}-\tilde{\epsilon}I_m$ where $\tilde\epsilon$ is small enough and $0< \tilde{\epsilon} < \min\{1/\lambda_1( \widetilde{\Sigma}_1 ), \dots, 1/\lambda_m ( \widetilde{\Sigma}_1)\}$.  Then
\[
\exp\bigl[-\tfrac{1}{2} (x- \phi^2(\mu_j))^\mathsf{T}(\widetilde\Sigma_j^1)^{-1}(x- \phi^2(\mu_j)) \bigr]  \leq \exp\bigl[-\tfrac{1}{2} (x- \phi^2(\mu_j))^\mathsf{T}(\widetilde\Sigma_j^2)^{-1}(x- \phi^2(\mu_j)) \bigr].
\]
Then from the above inequality and \eqref{eq-Sigma}, one has
\begin{align*}
\int_{\mathbb{R}^m}&\left\lvert \mathcal{N}_m(\phi^2(\mu_j),\widetilde\Sigma_j^1)- \mathcal{N}_m(\phi^2(\mu_j),\widetilde\Sigma_j^2)\right\rvert \, dx\\
&\leq  \frac{2^{m+1}}{(2\pi)^{m/2}}\int_{[0,\infty)^m}\biggl(  \frac{1}{\det(\widetilde\Sigma_j^1)^{1/2}}-\frac{1}{\det(\widetilde\Sigma_j^2)^{1/2}} \biggr) \exp\bigl[-\tfrac{1}{2} (x- \phi^2(\mu_j))^\mathsf{T}(\widetilde\Sigma_j^2)^{-1}(x- \phi^2(\mu_j)) \bigr] \,dx\\
&\leq 2^m\frac{\det(\widetilde\Sigma_j^2)^{1/2}-\det( \widetilde\Sigma_j^1)^{1/2}}{\det( \widetilde\Sigma_j^1)^{1/2}}.
\end{align*}

We divide the range of each of the eigenvalues $[h_n^2, M^2]$ into $L$ equidistant intervals and let $\lambda_{jl}=h_n^2(1+\epsilon/2^{m+3})^{2l_j/m}$ where $j=1,\dots, m$ and $ 1\leq l_j\leq L$. We pick $L$ to be the smallest integer which satisfies $h_n^2(1+\epsilon/2^{m+3})^{2l_j/m}\geq M^2$. We pick the $j$th eigenvalue of $\widetilde\Sigma_j^1$ and $\widetilde\Sigma_j^2$ to be in some interval $[\lambda_{j(l-1)},\lambda_{jl}]$ and satisfying the ordering on the eigenvalues. Then one has
\[
2^m\frac{\det(\Sigma_j^2)^{1/2}-\det(\Sigma_j^1)^{1/2}}{\det(\Sigma_j^1)^{1/2}}\leq 2^m \frac{\prod_{j=1}^m h_n^2(1+\epsilon/2^{m+3})^{2l_j/m}-\prod_{j=1}^m h_n^2(1+\epsilon/2^{m+3})^{2(l_j-1)/m} }{\prod_{j=1}^m h_n^2(1+ \epsilon/2^{m+3})^{(2l_j-1)/m} }\leq \frac{\epsilon}{8},
\]
and the metric entropy of $\boldsymbol{\lambda}$ in $[h_n^2, M^2]^m$ is bounded by
\[
\log N_{\boldsymbol\lambda} \leq m\biggl[\frac{\log(M^2/h_n^2)}{\log ( 1+\epsilon/2^{m+3})}+1\biggr].
\]

Letting $\bar{\theta}_n=\bar{\mu}_n$ and combining all the terms above on the entropy numbers, the metric entropy number of $D_{n,\epsilon}$ is bounded by
\[
\log N(\delta, D_{\delta, n})\leq   K_n\biggl[  1+\log\biggl( \frac{1+\epsilon/8}{\epsilon/8} \biggr) \biggr]
+ 2m \log \biggl[\frac{96 \bar{\theta}_n}{\epsilon  (2\pi\lambda_1(\widetilde{\Sigma}_1 ))^{1/2}} \biggr]
+ m\biggl[\frac{2\log (M/h_n)}{\log ( 1+ \epsilon / 2^{m+3})}+1 \biggr].
\]

Let $K_n=c_1\sqrt{n}$, $\bar{\theta}_n=c_2\sqrt{n}$, and $h_n=c_3n^{-1/b}$ for some constants $b>0$ and $c_1,c_2 $ and $c_3$ small enough. Then one has
\begin{align*}
\log N(\delta, D_{\delta, n})\leq n\beta.
\end{align*}
It remains to verify the condition that the prior mass outside $D_{\delta, n}$ is exponentially small given our priors on the parameters and the above choice of $K_n$, $\bar{\theta}_n$, $h_n$ and $M.$ We assume multivariate normal priors  (with diagonal covariance matrices) for $\theta$ and $\mu$, thus with a choice of $\bar{r}_n$, $\bar\mu_n$, one can show that  using changing of variables and Mill's inequality,  the tail $P(\lVert\theta\rVert\geq \bar{\theta}_n)$ decays exponentially. The elements $\sigma_1^2,\dots,\sigma_K^2, \sigma^2  $  and the variance (i.e., diagonal) terms of $\Sigma_j$, $j=1,\dots, m$, are assumed to follow the i.i.d.\ Gamma priors with density $b^{a}/\Gamma(a)x^{-a-1}\exp(-b/x)$ and hyperparameters $a$ and $b$. Then by direct calculation of the integrals,
\[
P(\min(\boldsymbol\sigma)\leq h_n)\leq c_4n\exp(-\tilde C_4h_n^{-b})\leq c_4\exp(-\tilde c_4n+\log n),
\]
which decays exponentially fast. 
\end{proof}

\section{Results on real and simulated data}
\label{data}

We illustrate the utility of the embedding using three data analysis
examples. The first example involves synthetic data with simple geometric structure to contrast the performance of our method with $k$-means clustering,
a Gaussian mixture model, and a Bayesian factor model.  In the second
example, compare the performance of our model with a logistic model, a
Gaussian mixture model, as well as a factor model on three supervised
classification problems from the UCI machine learning repository \cite{UCI}. The last example compares the spherical topic model we developed to
a latent Dirichlet allocation model on a corpus of NSF award abstracts \cite{NSF}.

\subsection{Line intersecting a plane}

Possibly the simplest example of a mixture of subspaces is a line puncturing a
plane. We will use this example to illustrate basic properties of the
mixture of subspaces as well as explore comparisons to comparable
models. We will study how well we can cluster the observed points into
those sampled from the plane or line respectively. The models we
compare are $K$-mans clustering ($K$-means), a mixture of normals
(GMM), a mixture of non-parametric factor models (MFM)
\cite{carvalho08}, our mixture of subspaces model with variable 
dimensions (MSM), and our mixture of subspaces model with dimension 
fixed to to $d=2$ (MSM $d=2.2$). For the subspace model we set the 
temperature parameter for the Gibbs posterior to $10^{-6}$, and 
acceptance rates between 38\% and 48\% were achieved for the 
subspace and affine mean parameters.
  
The mixture model for a line intersecting a plane in $\mathbb{R}^3$
comprises two components: a subspace $\mathbf{U}_1$ corresponding
to a line and a subspace $\mathbf{U}_2$ corresponding to the plane. 
Although simple, this example can be challenging situation to
infer. To understand the effect of uncertainty of subspace
measurements on accuracy of models we add a isotropic noise around
the subspaces via a precision parameter.

The data is specified by the following distribution with the following five values for the precision parameter of the isotropic noise around the subspaces, $\nu = [10, 5, 1, 0.5, 0.2]$:
\begin{align*}
&\text{\textsc{Line}}  & &\text{\textsc{Plane}} \\
U_{1} &\sim\operatorname{Unif}(\operatorname{V}(1,3)), &U_{2}  &\sim \operatorname{Unif}(\operatorname{V}(2,3)),\\
\mu_{1} &\sim \mathcal{N}_1(0,1),  &\mu_2 &\sim \mathcal{N}_2(0,I),\\
\Sigma^{-1}_{1}&\sim \operatorname{TGa}(1,1,\nu), &\operatorname{diag}(\Sigma^{-1}_{2})  &\iid \operatorname{TGa}(1,1,\nu),\\
(I - U_1U_1^\mathsf{T})^{-1}\theta_1  &\sim \mathcal{N}_3(0,I), & (I_3 - U_2U_2^\mathsf{T})^{-1}\theta_2  &\sim \mathcal{N}_3(0,I),
\end{align*}
where $\operatorname{TGa}(1,1,\nu)$ is a left truncated Gamma truncated at precision $\nu$. Given these parameters for the two mixture components we specify the following two conditional distributions
\begin{align*}
x \mid \text{Line} &\iid \mathcal{N}_3(U_1\mu_1 + \theta_1, U_1(\Sigma_1 -\sigma^2_1)U_1^\mathsf{T}+\sigma^2_1 I),\\
x \mid \text{Plane} &\iid \mathcal{N}_3(U_2\mu_2 + \theta_2, U_2(\Sigma_2 -\sigma^2_1 I)U_2^\mathsf{T}+\sigma^2_1 I).
\end{align*}
We generated $500$ observations from both the line and the plane, see Figure~\ref{fig2}. For each of the five variance levels, ten data sets were generated, and a holdout set of $50$ observations from the line and plane.

A comparison of clustering accuracy of the five models is summarized
in Table~\ref{simdata}. We report the range in clustering accuracy for
eaach method on the holdout set over the ten runs. We conclude from
Table~\ref{simdata} that: (1) $K$-means performs poorly, (2) as the
precision parameter increases the performance of the GMM improves and
starts to approach the MFM and MSM results, (3) the MSM with variable
dimension outperforms the MSM with fixed dimension, and (4) the MSM
and MFM results are very similar. Note that the MSM provides more
geometric information inculding the dimension of the subspace.


 \begin{table}[hbt]
  \caption{Range of cluster assignment accuracy for the simulated data
for the five methods.}
 \label{simdata}
\begin{center}
\begin{tabular}{|c|c|c|c|c|c|}
\multicolumn{4}{c}{\textsc{Synthetic Data}} \\
\cline{1-6}
Precision & $K$-means  & MSM  & MSM $d=2,2$ & GMM &  MFM  \\
\hline
$0.1$ 		&$(0.66,0.77)$ &  $(0.95,0.99)$   & $(0.92, 0.97)$  & $(0.89,0.97)$ & $(0.95, 0.99)$\\
$0.5$   		& $(0.64,0.82)$&   $(0.90,0.98)$ & $(0.90, 0.96)$ & $(0.87,0.98)$ &  $(0.91, 0.99)$ \\
$1$    		& $(0.57,0.70)$&   $(0.87,0.98)$ &  $(0.86, 0.97)$ & $(0.85,0.98)$ & $(0.88, 0.98)$ \\
$2$    		&  $(0.64,0.72)$& $(0.87,0.96)$  &  $(0.87, 0.95)$ & $(0.87,0.95)$   & $(0.86, 0.98)$\\
$5$     		& $(0.59,0.80)$	&    $(0.84,0.97)$ & $(0.84, 0.97)$  & $(0.84,0.97)$  &$(0.83, 0.97)$\\
\hline
\end{tabular}
\end{center}
\vspace{.05in}

\end{table}

\begin{figure}[h!]
\label{fig2}
\begin{center}
\includegraphics[scale=.45]{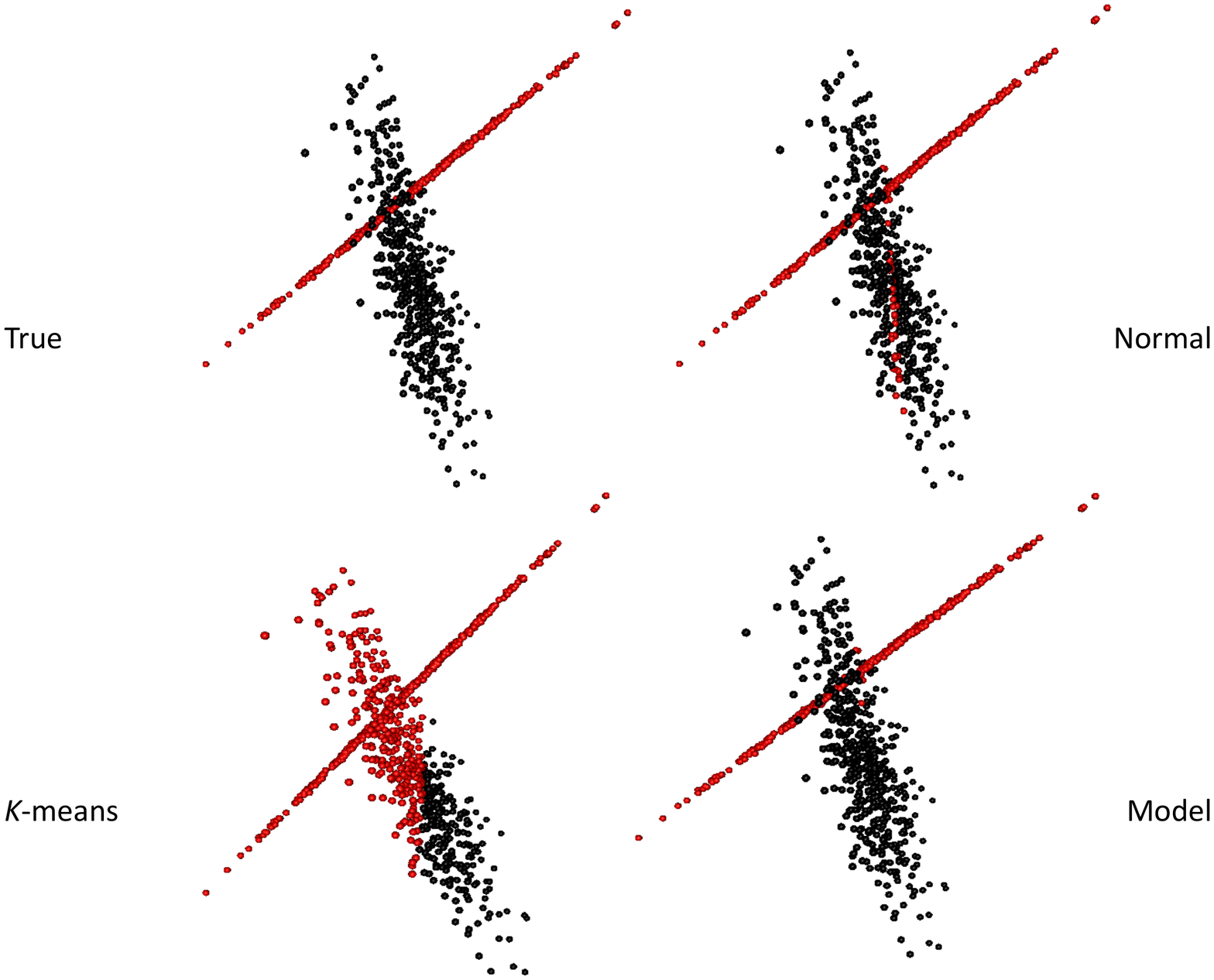}
\end{center}
\vspace{-.1in}
 \caption{The upper-left panel displays the true cluster assignments,
   the upper-right panel displays the assignemnts given by the GMM,
   the lower-left panel displays the the $K$-means assignment, and
the lower-right panel displays the assignmets by the MSM.}
 \end{figure}


\subsection{Classification on UCI data}

To study the clustering performance on more realistic data we examined
the classification accuracy on three data sets from the UCI Data 
Repository: the Statlog Vehicle Silhouettes data \cite{vehicle}, the 
Wisconsin Breast Cancer data \cite{breast}, and the Statlog Heart 
data \cite{heart}. Our metric of success on all three data was holdout
classification accuracy. We compared five models: a (multinomial)
logit model (Logit), our mixture of subspaces model with variable 
dimensions (MSM), our mixture of subspaces model with fixed 
dimensions (MSM $d=(5,2,2)$), and a mixture of non-parametric 
factor models (MFM) \cite{carvalho08}.

For the subspace models the temperature parameter of the Gibbs
posterior was set to obtain an acceptance ration in  in the range of 
20-40\% during the burn-in period. We did not use a cross-validation
criteria to set the temperature parameter due to computational burden.
To compute predictive accuracy we use the maximum a posterior estimate
of our MCMC runs to classify a new point.

The Heart Data Set contains $270$ observations of two classes with
$13$ covariates, the Vehicle Data Set contains $846$ observations of
four classes on $18$ covariates, and the Breast Cancer Data Set
contains $569$ observations of two classes on $30$ covariates. For
each dataset we measured the test error on a holdout set of 10\% of
the data. We repeated the test error estimates ten times and report
the range in test errors in Table~\ref{classres} for the results.

We conclude from Table~\ref{classres} that: the MSM with variable
dimension outperforms the MSM with fixed dimension, the mixture of
subspaces and mixture of factors perform as well or better than the
logit model which is the only supervised method, MSM and MFM have very
comparable performance.

\begin{table}[h]
\caption{Range of cluster assignment accuracy for the three data sets
  using five models on holdout data.}
\centering
\begin{tabular}{|c|c|c|c|c|c|}
\multicolumn{6}{c}{\textsc{Real Data}}\\
\cline{1-6}
Data Set & Logit & MSM & MFM &  MSM $d = (5,2,2)$ & GMM \\
\hline
Breast	& $(0.78, 0.86)$ 	& $(0.89, 0.94)$ & (0.90, 0.96) & (0.83, 0.89)	& $(0.64, 0.70)$ \\
Heart 	& $(0.72, 0.78)$ 	& $(0.77, 0.81)$ & (0.80, 0.82) & (0.73, 0.77)	& $(0.56, 0.60)$ \\
Vehicle 	& $(0.46, 0.59)$ 	& $(0.77, 0.83)$ & (0.76, 0.85) &  (0.75, 0.80)   & $(0.74, 0.79)$\\
\hline
\end{tabular}
\bigskip
\label{classres}
\vspace{.05in}

\end{table}

Our subspace models allows for an estimate of the dimension of the
linear subspace. This is not possible for either the Bayesian mixture models 
proposed in \cite{page13}, the non-parametric mixture of factor models
proposed in \cite{carvalho08}, or the penalized cost based mixture of
subspaces model \cite{LerZha2010}. In Table~\ref{dimvals} we state the
posterior probabilities estimates of the dimension of the subspaces.

\begin{table}[h!]
\caption{Posterior probabilities for the dimension of the subspace of clusters in each data set.}
 \label{dimvals}
\centering
\begin{tabular}{|ccc|}
\multicolumn{3}{c}{\textsc{Breast}} \\
\cline{1-3}
Class    & Dim. & Post. Prob.  \\
\hline
$1$     & $5$    & $0.45$     \\
$1$     & $6$    & $0.53$    \\
\hline
$2$       & $4$     & $0.66$     \\
$2$       & $5$     & $0.34$     \\
\hline
\end{tabular} \qquad 
\begin{tabular}{|ccc|}
\multicolumn{3}{c}{\textsc{Heart}} \\
\cline{1-3}
Class    & Dim. & Post. Prob.  \\
\hline
$1$     & $1$    & $0.10$  \\
$1$     & $2$    & $0.78$     \\
$1$     & $3$    & $0.12$    \\
\hline
$2$       & $1$     & $0.79$     \\
$2$       & $2$     & $0.21$     \\
\hline
\end{tabular} \quad
\begin{tabular}{|ccc|}
\multicolumn{3}{c}{\textsc{Vehicle}} \\
\cline{1-3}
Class    & Dim. & Post. Prob.  \\
\hline
$1$     & $1$    & $0.69$    \\
$1$     & $2$    & $0.31$    \\
\hline
$2$       & $1$     & $0.76$     \\
$2$       & $2$     & $0.24$     \\
\hline
$3$      & $1$     & $0.78$     \\
$3$       & $2$     & $0.22$     \\
\hline
$4$       & $1$     & $0.93$     \\
$4$       & $2$     & $0.07$     \\
\hline
\end{tabular}
\bigskip

\end{table}

\subsection{Analysis of NSF award abstracts}

In this subsection we compare the topic model proposed in 
Section \ref{topicmodels} to the standard latent Dirichlet allocation
model. The corpus we use to compare the two methods consists of 13,092 
abstracts from NSF awards in 2010 \cite{NSF}. The vocabulary was
constructed using the tokenizer from the Mallet package with 
bi-gram extraction \cite{MALLET}. The vocabulary was then reduced to
the to terms that were within the top $10\%$ term frequency�inverse document frequency
metric \cite{wu08} and occurred in at least five documents. The resulting
vocabulary consisted of $78,343$ terms. The average length of the 
documents after trimming the vocabulary was 379 words.

We compared the spherical topic model specified in \eqref{SAMmod} to the a standard LDA model with 20 topics.
We fixed the LDA model to have 20 topics. It was pointed out in \cite{reisinger10} that a direct comparison of topic
models and spherical topic models is not possible/meaningful. We instead examine the topic coherence and most 
relevant words in each topic. Example word clouds are displayed in Figures \ref{fig:ldacloud} and \ref{fig:samcloud}. 

One thing to notice about the output of the spherical topic models is that when investigating the topics, the positive and the negative parts of the topic vectors tend to be thematically coherent (also noted by \cite{reisinger10}). This is interesting because in a sense it allows a denser representation of the topics. One interesting result of the model specified in \eqref{SAMmod} 
on the NSF abstract data is that the positive and negative components of the topics tend to relate to broader impact terms 
and field or discipline terms respectively. It may not be surprising that writers of grants from different disciplines use different goals in broader impacts, however the spherical topic model structure gives us a tool for making thematic connections that would not be obvious from simply looking at the top terms of a topic.

Unlike standard admixture models, our spherical topic model allows for the inference on the number of topics. In Figure \ref{fig:ntopdist} we display the posterior distribution over the number of topics which is centered around $30$.

\begin{figure}
\includegraphics[scale = .1]{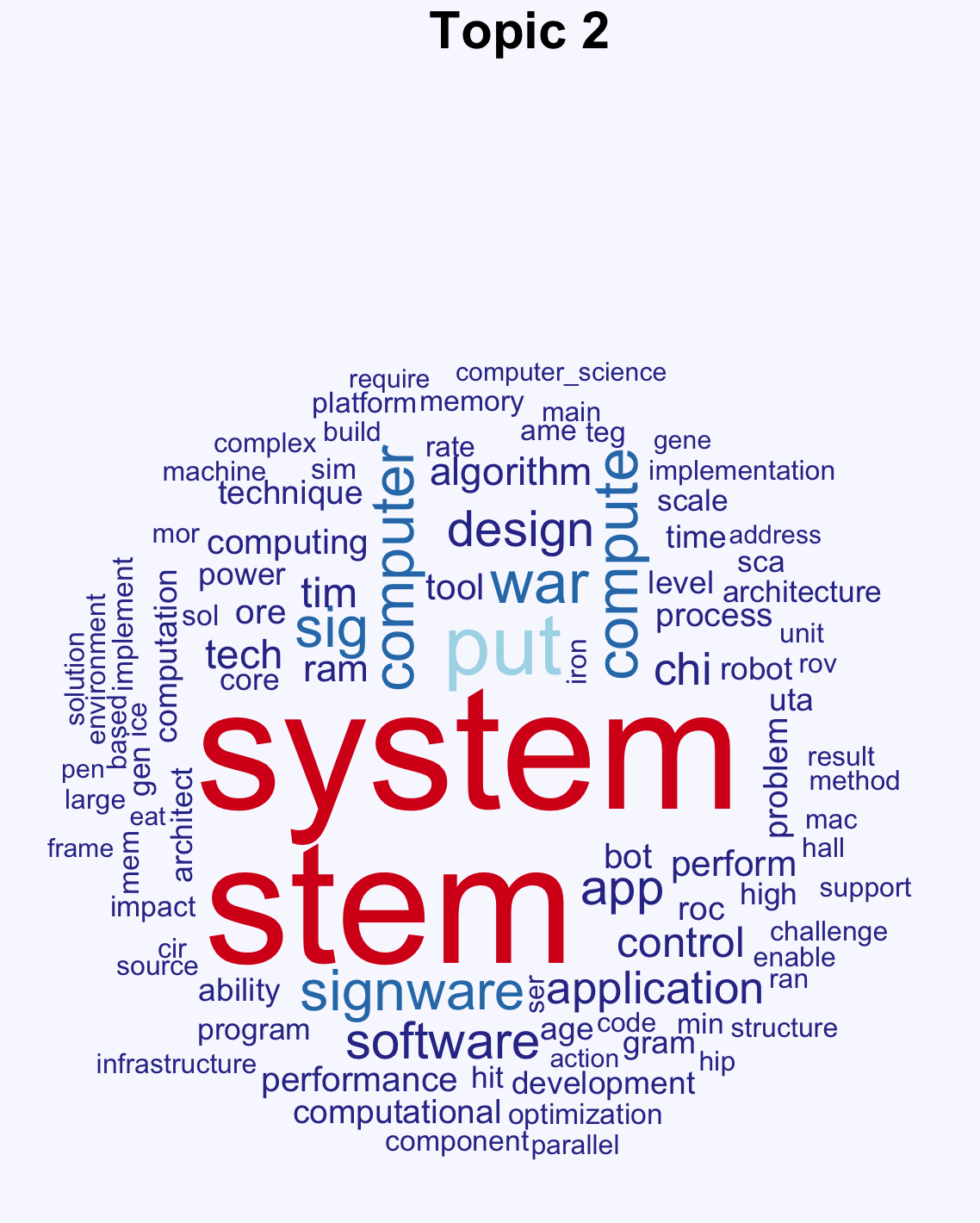}
\includegraphics[scale =.1]{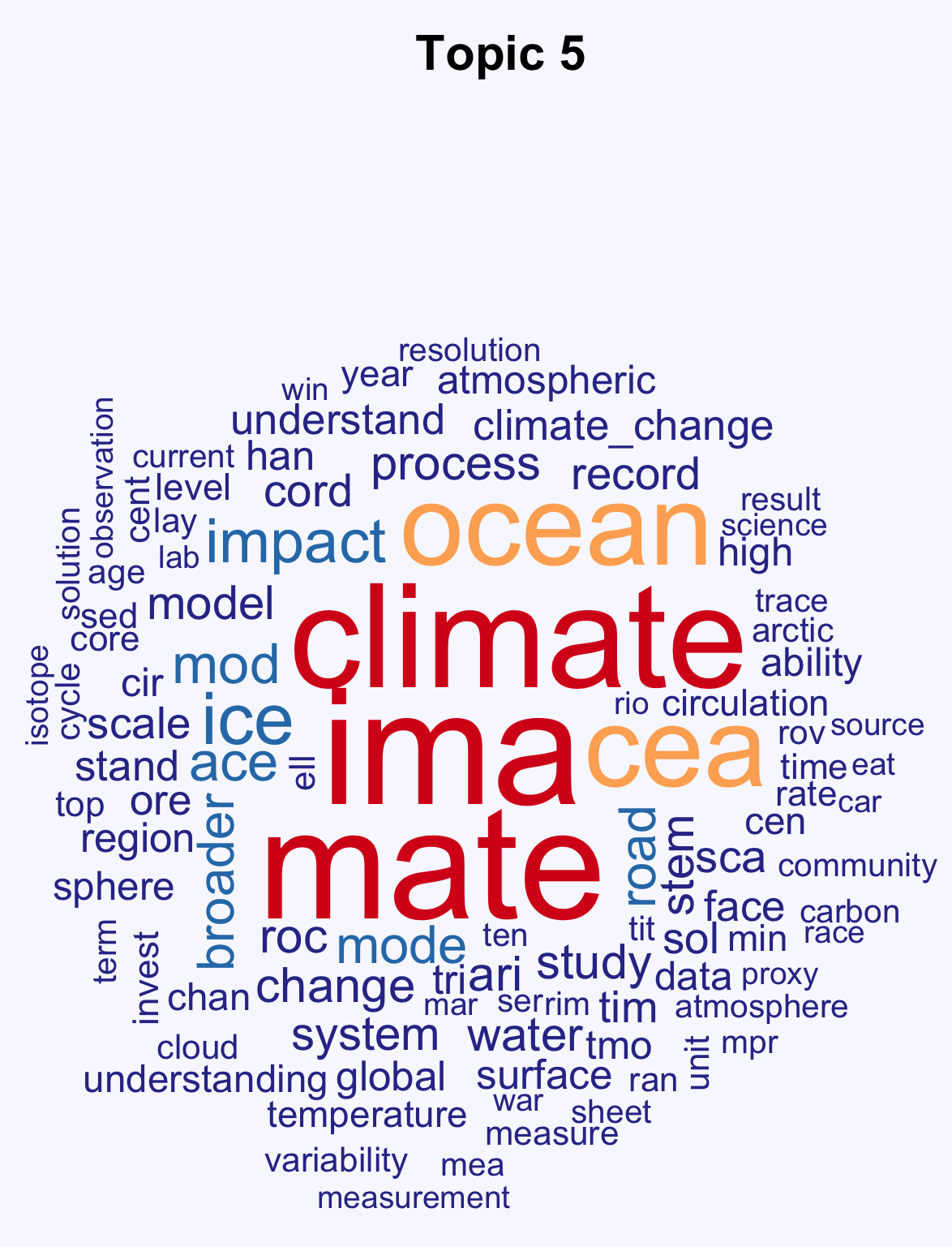}
\includegraphics[scale = .1]{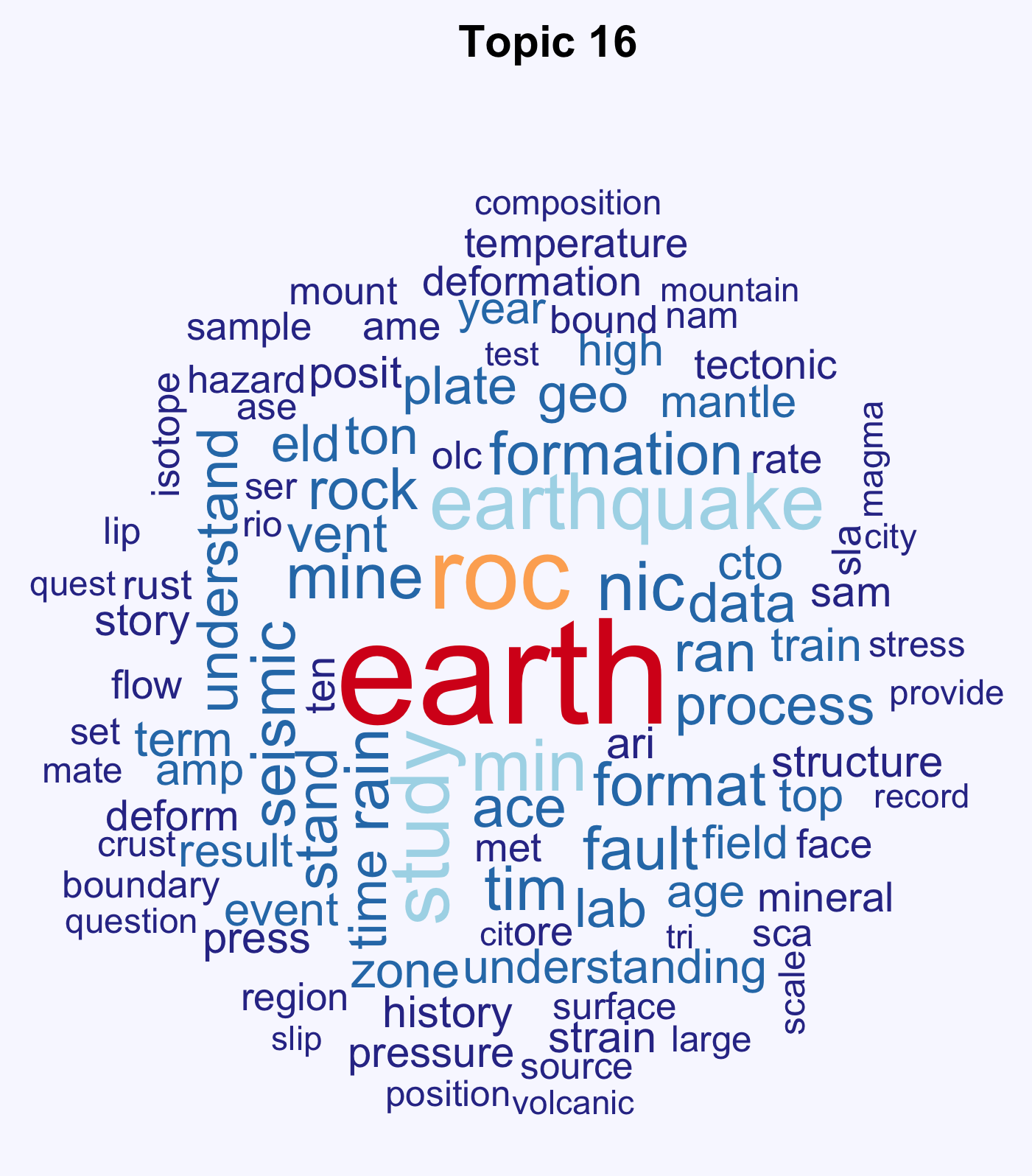}
\caption{Topics from LDA applied to 13,092 NSF awards. The size of the terms correspond to the value of the term in the topic.}
\label{fig:ldacloud}
\end{figure}

\begin{figure}[h!]
\includegraphics[scale = .1]{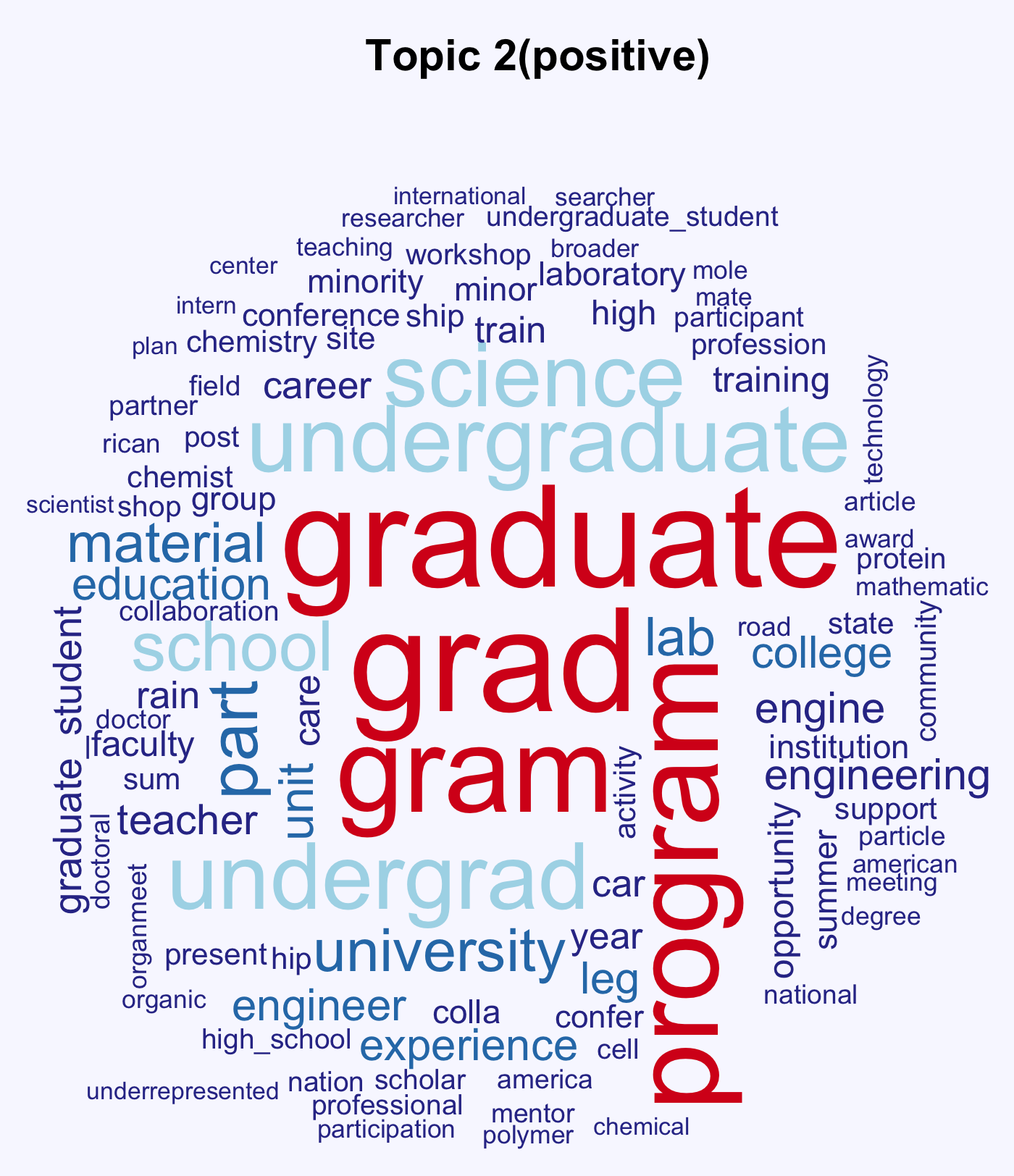}
\includegraphics[scale =.1]{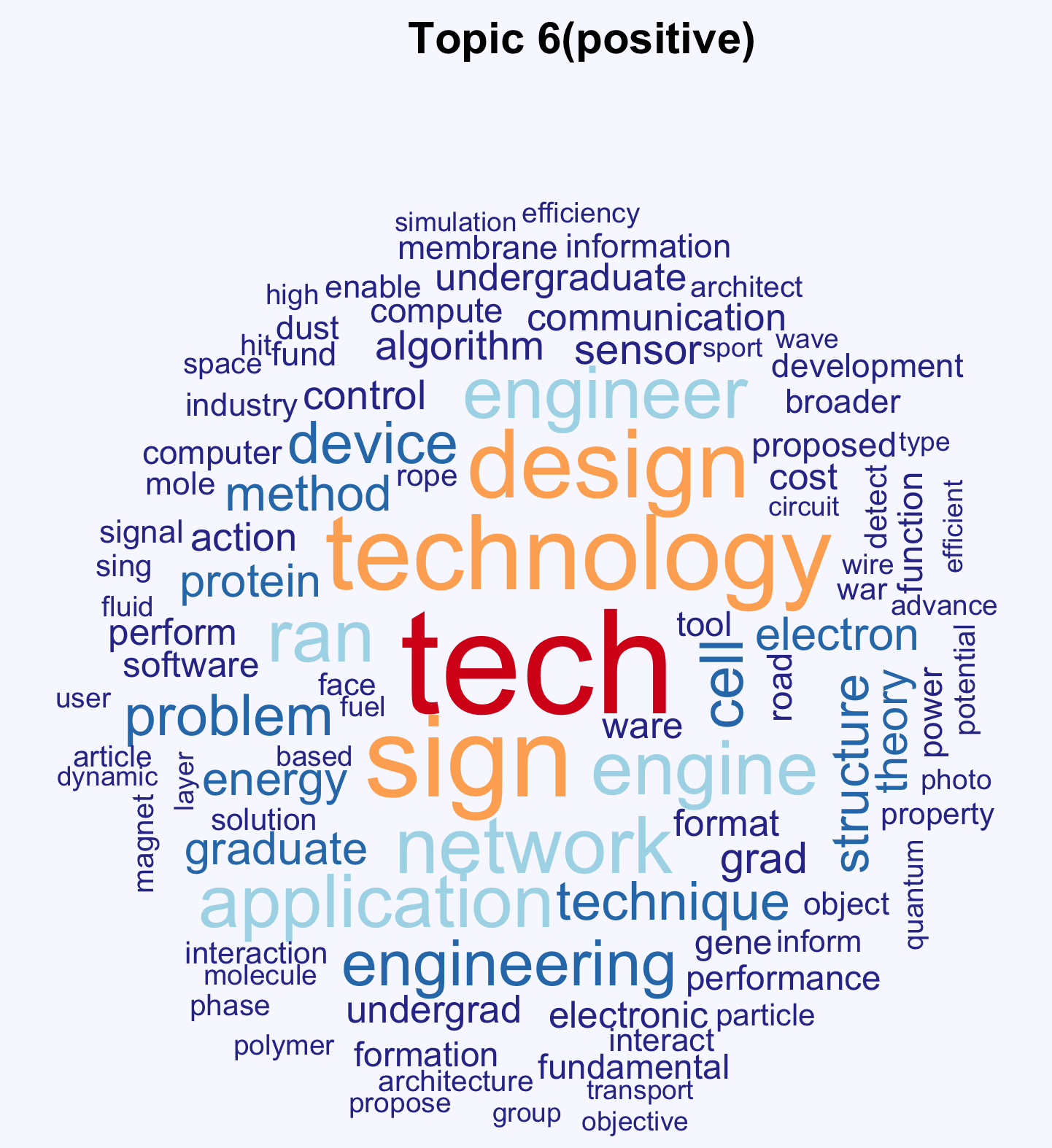}
\includegraphics[scale = .1]{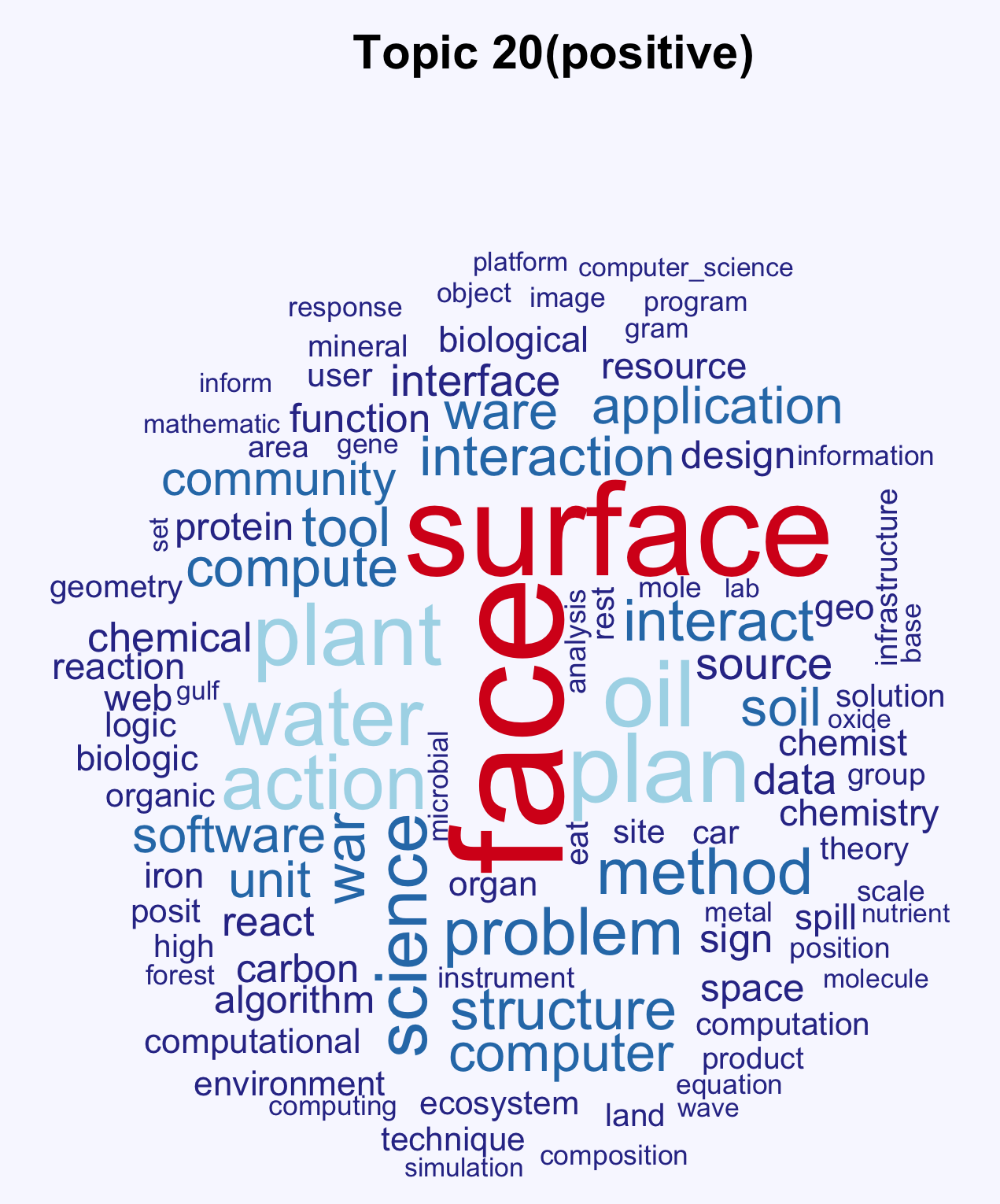}

\includegraphics[scale = .1]{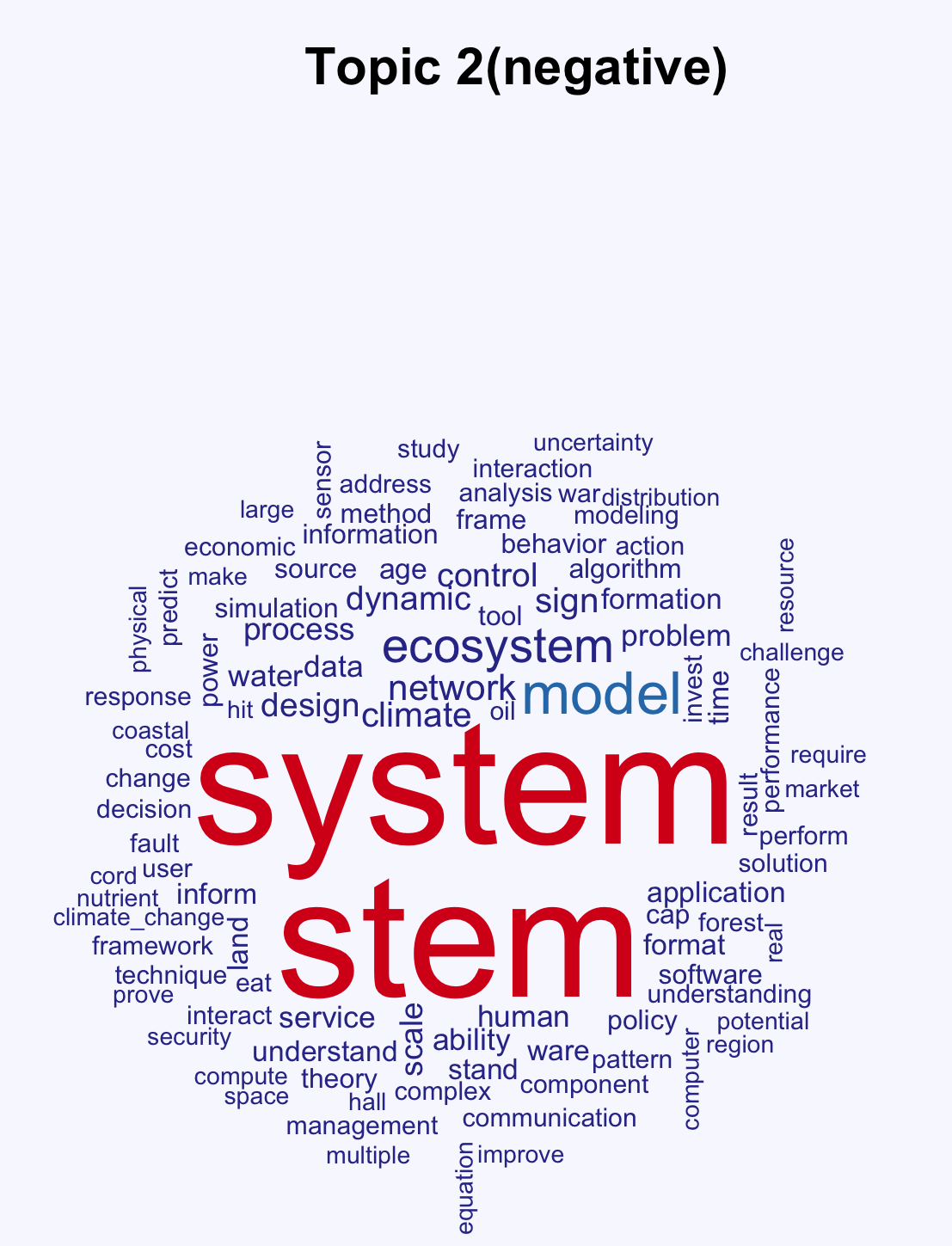}
\includegraphics[scale =.1]{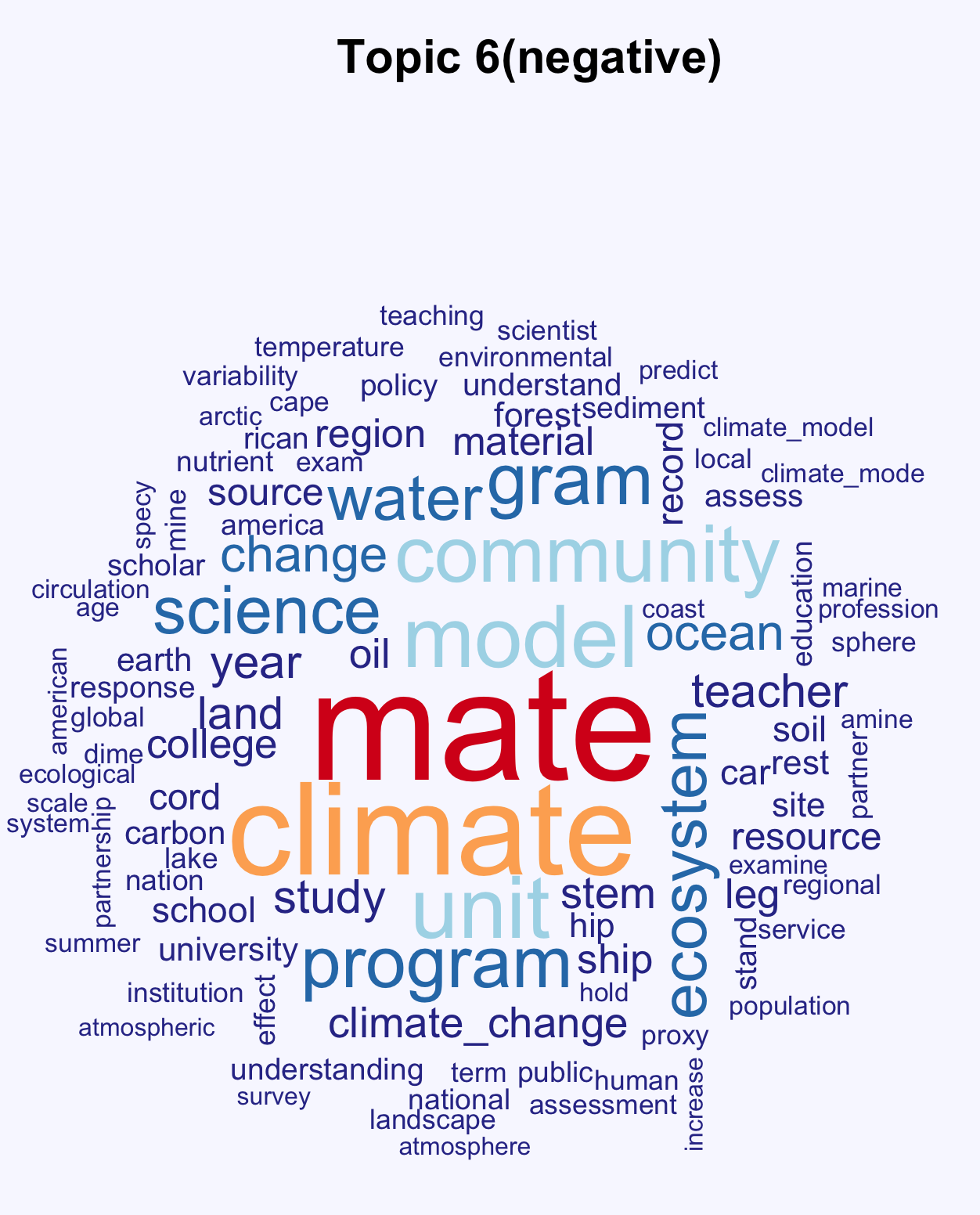}
\includegraphics[scale = .1]{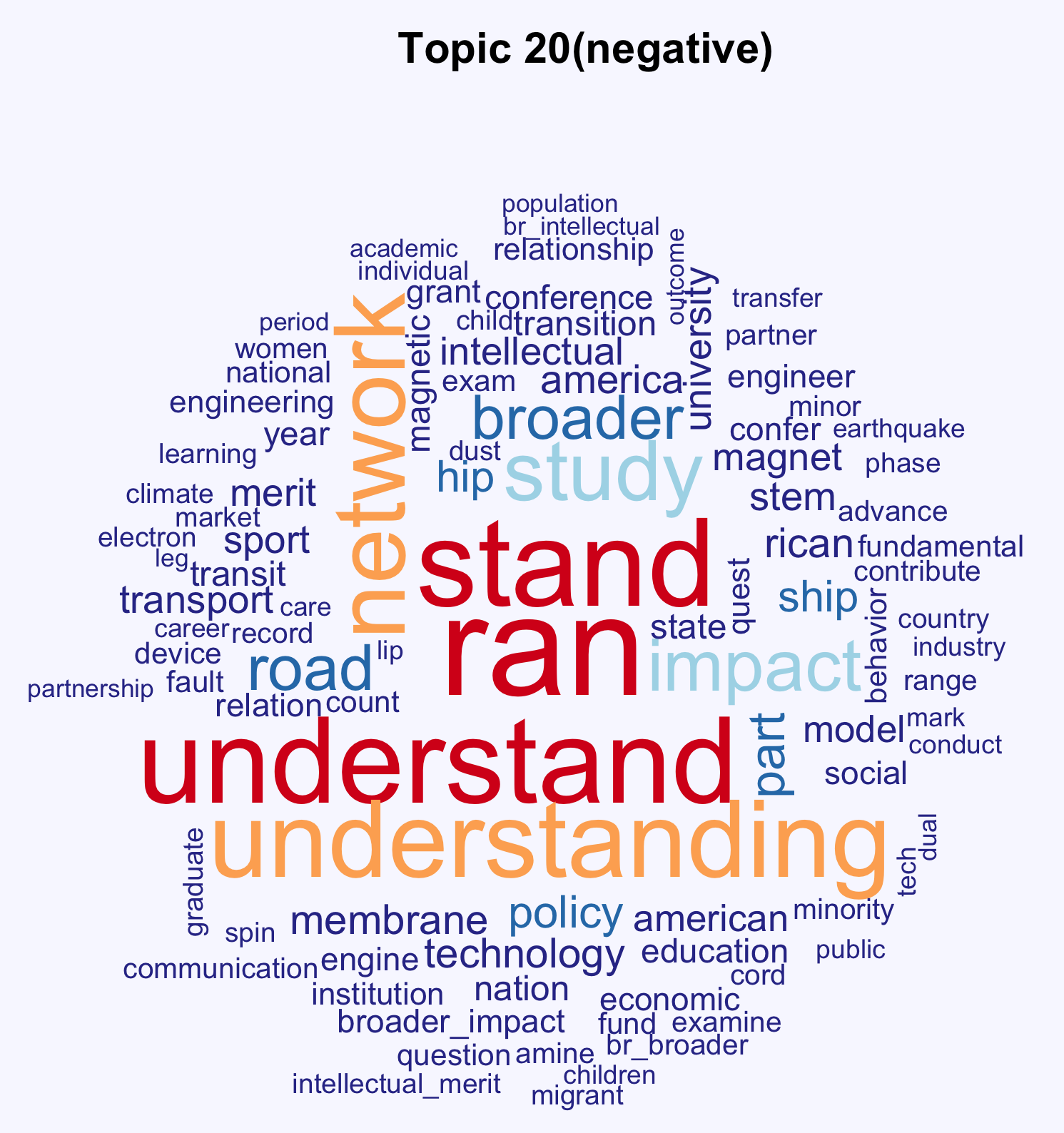}
\caption{Topics for the model specified in \eqref{SAMmod}. The size of the terms correspond to the absolute value of the terms in the topics. We display both the positive and negatively weighted topic values.}
\label{fig:samcloud}
\end{figure}

\begin{figure}
\begin{center}
\includegraphics[scale=.8]{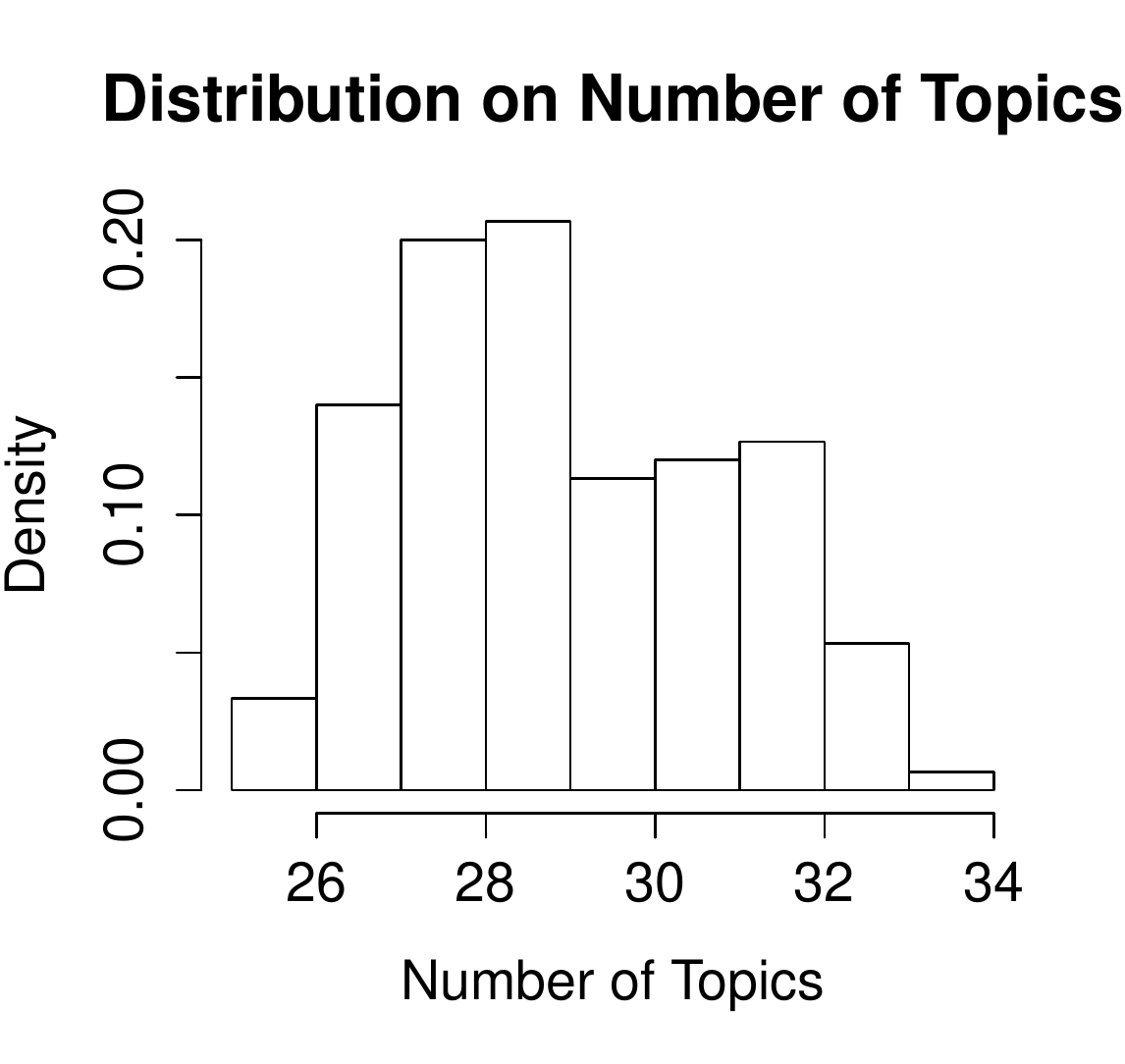}
\caption{Distribution on the number of topics inferred by the model. The uncertainty in the number of topics is greater than in the mixture of subspaces applications in the previous subsections.}
\label{fig:ntopdist}
\end{center}
\end{figure}


\section{Discussion}

We present a method for learning or inferring mixtures of linear subspaces of different dimensions. We show how this model can be trivially adapted for admixture modeling. The key idea in our procedure was using the observation that subspaces of different dimensions can be represented as points on a sphere is very useful for inference. The utility of this representation is that sampling from a sphere is straightforward, there exists a distance between subspaces of different dimensions that is differentiable and can be computed using principal angles, we avoid MCMC algorithms that jump between models of different dimensions. We suspect that this idea of embedding or representing models of different dimensions by embedding them into a common space with a distance metric that allows for ease of computation and sampling as well as nice analytic properties may also be of use in other settings besides subspaces.

Scaling our estimation procedure to higher dimensions and more samples will require greater computational efficiency and an EM-algorithm for this model holds promise. It is also of interest to examine if we can replace the Gibbs posterior with an efficient fully Bayesian procedure.

\section*{Acknowledgements}
SM and BST would like to thank Robert Calderbank, Daniel Runcie, and Jesse Windle for useful discussions. SM is pleased to acknowledge support from grants NIH (Systems Biology) 5P50-GM081883, AFOSR FA9550-10-1-0436, and NSF CCF-1049290. BST is pleased to acknowledge support from NSF grant DMS-1127914 to the Statistics and Applied Mathematics Institute. The work of LHL is partially supported by AFOSR FA9550-13-1-0133, NSF DMS-1209136, and NSF DMS-1057064. This work of LL  is supported by  a DUKE iiD grant and grant R01ES017240 from the National Institute of Environmental  Health Sciences (NIEHS) of the National Institute of Health (NIH).


\end{document}